\documentclass[12pt]{article}

\usepackage{graphicx, amsmath, amssymb, amsthm, multicol, bbold, float, subcaption,verbatim}
\usepackage[utf8x]{inputenc}
\usepackage{epstopdf}

\usepackage{cite}

\usepackage{color}   
\usepackage{hyperref}
\hypersetup{
	hidelinks,
    linktoc=all     
}

\newtheorem{definition}{Definition}[section]
\newtheorem{theorem}[definition]{Theorem}
\newtheorem{lemma}[definition]{Lemma}
\newtheorem{corollary}[definition]{Corollary}
\newtheorem{proposition}[definition]{Proposition}
\theoremstyle{definition}

\newtheorem{example}[definition]{Example}

\numberwithin{equation}{section}

\newcommand{\cl}{\operatorname{cl}}
\newcommand{\interior}{\operatorname{int}}
\newcommand{\conv}{\operatorname{conv}}
\newcommand{\cone}{\operatorname{cone}}
\newcommand{\proj}{\operatorname{proj}}

\DeclareMathOperator*{\argmin}{arg\,min}
\newcommand{\trans}[1]{#1^{\mathsf{T}}}

\begin{document}
\title{Convex Projection and Convex Multi-Objective Optimization}
\author{Gabriela Kov\'a\v{c}ov\'a\thanks{Vienna University of Economics and Business, Institute for Statistics and Mathematics, Vienna A-1020, AUT, gabriela.kovacova@wu.ac.at and birgit.rudloff@wu.ac.at.} \and Birgit Rudloff\footnotemark[1] 
}
\maketitle
\abstract{
In this paper we consider a problem, called convex projection, of projecting a convex set onto a subspace. We will show that to a convex projection one can assign a particular multi-objective convex optimization problem, such that the solution to that problem also solves the convex projection (and vice versa), which is analogous to the result in the polyhedral convex case considered in~\cite{LW16}. In practice, however, one can only compute approximate solutions in the (bounded or self-bounded) convex case, which solve the problem up to a given error tolerance. We will show that for approximate solutions a similar connection can be proven, but the tolerance level needs to be adjusted. That is, an approximate solution of the convex projection solves the multi-objective problem only with an increased error. Similarly, an approximate solution of the multi-objective problem solves the convex projection with an increased error. In both cases the tolerance is increased proportionally to a multiplier. These multipliers are deduced and shown to be sharp.
These results allow to compute approximate solutions to a convex projection problem by computing approximate solutions to the corresponding multi-objective convex optimization problem, for which algorithms exist in the bounded case.

For completeness, we will also investigate the potential generalization of the following result to the convex case. In~\cite{LW16}, it has been shown for the polyhedral case, how to construct a polyhedral projection associated to any given vector linear program and how to relate their solutions. This in turn yields an equivalence between polyhedral projection, multi-objective linear programming and vector linear programming. We will show that only some parts of this result can be generalized to the convex case, and discuss the limitations.
\\[.2cm]
\\[.2cm]
}

\section{Introduction}
Let us start with a short motivation, describing a field of research where convex projection problems, the main object of this paper, arise.
A dynamic programming principle, called a set-valued Bellman principle, for a particular vector optimization problem has been proposed in~\cite{KR21}, and has also been applied in~\cite{KRC20,FRZ21} to other multivariate problems. Often, dynamic optimization problems depend on some parameter, in the above examples the initial capital. Then, the dynamic programming principle leads to a sequence of parametrized vector optimization problems, that need to be solved recursively backwards in time. A parametrized vector optimization problem can equivalently be written as a projection problem. If the problem happens to be polyhedral, the theory of \cite{LW16} can be applied, which allows to rewrite the polyhedral projection as a particular multi-objective linear program, and thus, known solvers can be used to solve the original parametrized vector optimization problem for all values of the parameter. This method has been used e.g. in~\cite{KRC20}. However, in general the problem is convex, and not necessarily polyhedral, and thus leads to a convex projection problem.

The aim of this paper is to develop the needed theory to treat these problems. That is, to define the problem of convex projection, find appropriate solution concepts and, in analogy to the polyhedral case, relate those to a multi-objective convex problem, for which, at least in the bounded case, solvers are already available that compute approximate solutions, see~\cite{ESS11, LRU14, DLSW21}. It is left for future research to develop algorithms to solve also unbounded multi-objective convex problems, which would, by the results of this paper, also allow to solve certain unbounded convex projection problems. This, in turn, would then enable to solve the above mentioned problems arising from multivariate dynamic programming as they are by construction unbounded. Thus, the present research can be seen as a first step into that direction, but might also be of independent interest.

Let us start with a short description of the known results in the polyhedral case.
Polyhedral projection is a problem of projecting a polyhedral convex set, given by a finite collection of linear inequalities, onto a subspace. This problem is considered in~\cite{LW16}, where an equivalence between polyhedral projection, multi-objective linear programming and vector linear programming is proven. 

The following problem is called a \textit{polyhedral projection:} 
\begin{align*}
\text{compute } Y = \{ y \in \mathbb{R}^m \; \vert \; \exists x \in \mathbb{R}^n: \; Gx + Hy \geq h \}, 
\end{align*}
where the matrices $G \in \mathbb{R}^{k \times n}, H \in \mathbb{R}^{k \times m}$ and a vector $h \in \mathbb{R}^k$ define a polyhedral feasible set $S = \{ (x,y) \in \mathbb{R}^n \times \mathbb{R}^m \; \vert \; Gx + Hy \geq h \}$ to be projected. To the above polyhedral projection corresponds a multi-objective linear program
\begin{align*}
\min \begin{pmatrix}
y \\ -\trans{\mathbf{1}} y
\end{pmatrix} \text{ with respect to } \leq_{\mathbb{R}^{m+1}_+} \text{ subject to } Gx + Hy \geq h.
\end{align*}
In~\cite{LW16} it is proven that every solution of the associated multi-objective linear program is also a solution of the polyhedral projection 
(and vice versa)
and, moreover, a solution exists whenever the polyhedral projection is feasible. This enables to solve polyhedral projection problems using the existing solvers for multi-objective (vector) linear programs such as~\cite{LW17}. Furthermore,~\cite{LW16} shows how to construct a polyhedral projection associated to any given vector linear program. A solution of the given vector linear program can then, under some assumptions, be recovered from a solution of the associated polyhedral projection. This yields the aforementioned equivalence between polyhedral projection, multi-objective linear programming and vector linear programming.

In this paper we are interested in a projection problem, where the feasible set $S \subseteq \mathbb{R}^n \times \mathbb{R}^m$ is convex, but not necessarily polyhedral. Our problem, which we call a \textit{convex projection}, is of the form 
\begin{align}
\label{CP}
\tag{CP}
\text{compute } Y  = \{ y \in \mathbb{R}^m \; \vert \; \exists x \in \mathbb{R}^n: \; (x, y) \in S  \}.
\end{align}
Under computing the set $Y$ we understand computing a collection of feasible points that (exactly or approximately) span the set $Y$.
The question we want to answer is whether there exists a similar connection between convex projection, multi-objective convex optimization and convex vector optimization as there is in the linear (polyhedral) case. Inspired by~\cite{LW16}, we consider the multi-objective convex problem 
\begin{align}
\label{MOCP-1}
\min 
\begin{pmatrix}
y \\
- \trans{\mathbf{1}} y
\end{pmatrix} \text{ with respect to } \leq_{\mathbb{R}^{m+1}_+} \text{ subject to } (x, y) \in S.
\end{align}
Just as in the polyhedral case, there is a close link between the problems~\eqref{CP} and~\eqref{MOCP-1}. Relating solutions of the two problems, however, becomes somehow more involved. Namely, we need to make clear with which solution concept we work and also consider the boundedness of the problem. The notion of an exact solution does exist, but it is mostly a theoretical concept as it usually does not consist of finitely many points. In practice, instead, one obtains approximate solutions (i.e.~finite $\epsilon$-solutions), which solve the problem up to a given error tolerance. We consider exact solutions, as well as approximate solutions of bounded and self-bounded problems. Precise definitions will be given in the following sections. The results for exact solutions (which does not assume (self-)boundedness) are parallel to those for the linear case considered in~\cite{LW16}, we obtain an equivalence between exact solutions of the convex projection~\eqref{CP} and the associated multi-objective problem~\eqref{MOCP-1}. For approximate solutions (of bounded or self-bounded problems) a connection also exists, but the tolerance needs to be adjusted. That is, an approximate solution of the convex projection solves the multi-objective problem only with an increased error. Similarly, an approximate solution of the multi-objective problem solves the convex projection with an increased error. In both cases the tolerance is increased proportionally to a multiplier that is roughly equal to the dimension of the problem.

A problem similar to~\eqref{CP} was considered in~\cite{SZC18}, where an outer approximation Benson algorithm was adapted to solve it under some compactness assumptions. Here, we take a more theoretical approach to relate the different problem classes. 
For this reason we also consider a general convex vector optimization problem. There does exist an associated convex projection, which spans (up to a closure) the upper image of the vector problem. However, instead of having a correspondence between solutions of the two problems, we can only connect (exact or approximate) solutions of the associated projection to (exact or approximate) infimizers of the vector optimization.

The paper is organized as follows: Section~\ref{sec_notation} contains preliminaries and introduces the notation, Section~\ref{sec_CVOP} summarizes the solution concepts and properties of a convex vector optimization problem. In Section~\ref{sec_convexprojection}, we define a convex projection and introduce the corresponding solution concepts. Section~\ref{sec_CP_MOCP} contains the main results, the connection between the convex projection and an associated multi-objective problem is established. Section~\ref{sec_CVOP_CP} formulates a convex projection corresponding to a given convex vector optimization problem. Finally, in Section~\ref{appendix_DLSW21} the derived results are considered under a slightly different solution concept based on the Hausdorff distance, which was recently proposed in~\cite{DLSW21}.

\section{Preliminaries}
\subsection{Notation}

\label{sec_notation}

For a set $A \subseteq \mathbb{R}^m$ we denote by $\conv A, \cone A, \cl A$ and $\interior A$ the convex hull, the conical hull, the closure and the interior of $A$, respectively. The recession cone of the set $A \subseteq \mathbb{R}^m$ is  $A_{\infty} = \{ y \in \mathbb{R}^m \; \vert \; \forall x \in A, \lambda \geq 0 : \; x + \lambda y \in A \}$.  
A cone $C \subseteq \mathbb{R}^m$ is \textit{solid} if it has a non-empty interior, it is \textit{pointed} if $C \cap -C = \{0\}$ and it is \textit{non-trivial} if $\{0\} \subsetneq C \subsetneq \mathbb{R}^m$. A non-trivial pointed convex cone $C\subseteq \mathbb{R}^m$ defines a partial ordering on $\mathbb{R}^m$ via $q^{(1)} \leq_C q^{(2)}$ if and only if $q^{(2)} - q^{(1)} \in C$. Given a non-trivial pointed convex cone $C\subseteq \mathbb{R}^m$ and a convex set $\mathcal{X} \subseteq \mathbb{R}^n$, a function $\Gamma: \mathcal{X} \rightarrow \mathbb{R}^m$ is $C$-convex, if for all $x^{(1)}, x^{(2)} \in \mathcal{X}$ and $\lambda \in [0,1]$ it holds $\Gamma(\lambda x^{(1)} + (1-\lambda) x^{(2)}) \leq_C \lambda \Gamma (x^{(1)}) + (1-\lambda) \Gamma(x^{(2)})$. The Minkowski sum of two sets $A, B \subseteq \mathbb{R}^m$ is, as usual, denoted by $A + B = \{ a + b \; \vert \; a \in A, b \in B\}$. As a short hand notation we will denote the Minkowski sum of a set $A \subseteq \mathbb{R}^m$ and the negative of a cone $C \subseteq \mathbb{R}^m$ by $A - C = A + (-C)$.

Throughout the paper we fix $p \in [1, \infty]$ and denote by $\Vert \cdot \Vert$ the $p$-norm and by $B_{\epsilon} = \{ z \in \mathbb{R}^m \; \vert \; \Vert z \Vert \leq \epsilon \}$ the closed $\epsilon$-ball around the origin in this norm. The standard basis vectors are denoted by $e^{(1)}, \dots, e^{(m)}$. In the following, we work with Euclidean spaces of various dimensions, mainly with $\mathbb{R}^m$ and $\mathbb{R}^{m+1}$. In order to keep the notation as simple as possible we do not explicitly denote the dimensions of most vectors (e.g.~$e^{(1)}$) or sets (e.g.~$B_{\epsilon}$), as they should be clear from the context. One exception is the vector of ones, where we denote $\mathbf{1} \in \mathbb{R}^{m}$ and $\mathbb{1} \in \mathbb{R}^{m+1}$ to avoid any confusion.

We use three projection mappings, $\proj_y: \mathbb{R}^{n} \times \mathbb{R}^m \rightarrow \mathbb{R}^{m}, \proj_x: \mathbb{R}^{n} \times \mathbb{R}^m \rightarrow \mathbb{R}^{n}$ and $\proj_{-1}: \mathbb{R}^{m+1} \rightarrow \mathbb{R}^{m}$. The mapping $\proj_y$, given by the matrix $\proj_y = \begin{pmatrix}
0 & I
\end{pmatrix},$ projects a vector $(x, y) \in \mathbb{R}^{n} \times \mathbb{R}^m$ onto $y \in \mathbb{R}^m$. The mapping $\proj_x$, given by the matrix $\proj_x = \begin{pmatrix}
I & 0
\end{pmatrix},$ projects a vector $(x, y) \in \mathbb{R}^{n} \times \mathbb{R}^m$ onto $x \in \mathbb{R}^n$. The mapping $\proj_{-1}$ drops the last element of a vector from $\mathbb{R}^{m+1}$.

\subsection{Convex Vector Optimization Problem}
\label{sec_CVOP}

In this section we recall the definition of a convex vector optimization problem, its properties and solution concepts from~\cite{LRU14, U18}, which are adopted within this work, and which are based on the lattice approach to vector optimization~\cite{L11}. 

A convex vector optimization problem is
\begin{align}
\label{CVOP}
\tag{CVOP}
\min \Gamma (x) \text{ with respect to } \leq_C \text{ subject to } x \in \mathcal{X},
\end{align}
where $C \subseteq \mathbb{R}^m$ is a non-trivial, pointed, solid, convex ordering cone, $\mathcal{X} \subseteq \mathbb{R}^n$ is a convex set and the objective function $\Gamma: \mathcal{X} \rightarrow \mathbb{R}^m$ is $C$-convex. The convex feasible set $\mathcal{X}$ is usually specified via a collection of convex inequalities. 
A convex vector optimization problem is called a multi-objective convex problem if the ordering cone is the natural ordering cone, i.e. if $C=\mathbb{R}^m_+$. A particular multi-objective convex problem that helps in solving a convex projection problem will be considered in Section~\ref{sec_CP_MOCP}.
The general convex vector optimization problem and its connection to convex projections will be treated in Section~\ref{sec_CVOP_CP}.

The image of the feasible set $\mathcal{X}$ is defined as $\Gamma[\mathcal{X}] = \{ \Gamma(x) \; \vert \; x \in\mathcal{X} \}$. The closed convex set 
\begin{align*}
\mathcal{G} = \cl ( \Gamma[\mathcal{X}] + C )
\end{align*}
is called the \textbf{upper image} of~\eqref{CVOP}.
A feasible point $\bar{x} \in \mathcal{X}$ is a \textbf{minimizer} of~\eqref{CVOP} if it holds $(\{\Gamma (\bar{x})\} - C \backslash \{0\}) \cap \Gamma[\mathcal{X}] = \emptyset$. It is a \textbf{weak minimizer} if $(\{\Gamma (\bar{x})\} - \interior C) \cap \Gamma[\mathcal{X}] = \emptyset$.

\begin{definition}[see~\cite{U18}]
\label{def_CVOP_bounded}
The problem~\eqref{CVOP} is \textbf{bounded} if for some $q \in \mathbb{R}^m$ it holds $\mathcal{G} \subseteq \{q\} + C$. If~\eqref{CVOP} is not bounded, it is called \textbf{unbounded}.  The problem \eqref{CVOP} is \textbf{self-bounded} if $\mathcal{G} \neq \mathbb{R}^m$ and for some $q \in \mathbb{R}^m$ it holds $\mathcal{G} \subseteq \{q\} + \mathcal{G}_{\infty}$. 
\end{definition}
Self-boundedness is related to the tractability of the problem, i.e., to the existence of polyhedral inner and outer approximations to the Pareto frontier of~\eqref{CVOP} such that the Hausdorff distance between the two is finite, see~\cite{U18}. Self-boundedness allows to turn an unbounded problem into a bounded one by replacing the ordering cone $C$ by $\mathcal{G}_{\infty}$.
One can notice that a bounded problem is a special case of a self-bounded one. The following relationship holds.
\begin{lemma}
\label{lemma_self-bd}
A self-bounded~\eqref{CVOP} is bounded if and only if $\mathcal{G}_{\infty} = \cl C$.
\end{lemma}
\begin{proof}
Boundedness of~\eqref{CVOP} implies $\mathcal{G}_{\infty} \subseteq \cl C$, together with $ \cl C\subseteq\mathcal{G}_{\infty}$ one obtains $\mathcal{G}_{\infty} = \cl C$.
For the reverse, note that self-boundedness and $\mathcal{G}_{\infty} = \cl C$ imply $\mathcal{G} \subseteq \{q\} + \cl C$. To prove the boundedness of~\eqref{CVOP}, one needs to show that also $\mathcal{G} \subseteq \{\bar{q}\} + C$ for some $\bar{q} \in \mathbb{R}^m$. This follows from $\cl C \subseteq \{-\delta c\} + C$ for arbitrary $\delta >0$ and $c \in \interior C$, which holds as $C$ is a solid convex cone. Thus one can set $\bar{q} = q-\delta c$.
  \end{proof}
Boundedness of the problem~\eqref{CVOP} is defined as boundedness of the upper image $\mathcal{G}$ with respect to the ordering cone $C$. According to the above proof, for a solid cone $C$ boundedness of $\mathcal{G}$ with respect to $C$, with respect to $\cl C$ and with respect to $\interior C$ are equivalent.

(Self-)boundedness is important for the definition of approximate solutions.
In the following, a normalized direction $c \in \interior C$, i.e.~$\Vert c \Vert = 1$, and a tolerance $\epsilon > 0$ are fixed.
\begin{definition}[see~\cite{LRU14, U18}]
\label{def_CVOP_sol}
A set $\bar{X} \subseteq \mathcal{X}$ is an \textbf{infimizer} of~\eqref{CVOP} if it satisfies 
\begin{align*}
\mathcal{G} = \cl \conv (\Gamma [\bar{X}] + C ).
\end{align*}
An infimizer $\bar{X}$ of~\eqref{CVOP} is called a \textbf{(weak) solution} if it consists of (weak) minimizers only.
A nonempty finite set $\bar{X} \subseteq \mathcal{X}$ is a \textbf{finite $\epsilon$-infimizer} of a bounded problem~\eqref{CVOP} if
\begin{align}
\label{eq16}
\mathcal{G} \subseteq \conv \Gamma [\bar{X}] + C -\epsilon \{c\}.
\end{align}
A nonempty finite set $\bar{X} \subseteq \mathcal{X}$ is a \textbf{finite $\epsilon$-infimizer} of a self-bounded problem~\eqref{CVOP} if 
\begin{align}
\label{eq17}
\mathcal{G} \subseteq \conv \Gamma [\bar{X}] + \mathcal{G}_{\infty} -\epsilon \{c\}.
\end{align}
A finite $\epsilon$-infimizer $\bar{X}$ of~\eqref{CVOP} is called a \textbf{finite (weak) $\epsilon$-solution} if it consists of (weak) minimizers only.
\end{definition}

There are algorithms such as~\cite{ESS11, LRU14, DLSW21} for finding finite $\epsilon$-solutions of a bounded \eqref{CVOP} (under some assumptions, e.g.~compact feasible set, continuous objective). The definition of a finite $\epsilon$-solution in the self-bounded case is more of a theoretical concept, as it assumes that the recession cone of the upper image is known. Nevertheless, we consider also this case, as it can be treated jointly with the bounded case. To the best of our knowledge, there is, so far, no algorithm for solving an unbounded~\eqref{CVOP}, which also means, there is no definition yet of an approximate solution in the not self-bounded case.

The definition of a finite $\epsilon$-solution in~\cite{DLSW21} differs slightly from the one used here that is based on~\cite{LRU14}. 
However, the structure of the main results (Theorems~\ref{thm_CPsolMOCP},~\ref{thm_MOCPsolCP},~\ref{thm_sol_CVOP1} and~\ref{thm_sol_CVOP2}) remain valid using the definition from~\cite{DLSW21} under some minor adjustments of some details, which are given in Section~\ref{appendix_DLSW21}. 

In this paper we will often refer to a solution as an exact solution and to a finite $\epsilon$-solution as an approximate solution.

\section{Convex Projections}
\subsection{Definitions}
\label{sec_convexprojection}
The object of interest is the convex counterpart of the polyhedral projection, that is, the problem of projecting a convex set onto a subspace. Thus, let a convex set  $S \subseteq \mathbb{R}^n \times \mathbb{R}^m$ be given. Just as the feasible set $\mathcal{X}$ of~\eqref{CVOP}, it could be specified via a collection of convex inequalities. The aim is to project the feasible set $S$ onto its $y$-component, that is to  
\begin{align}
\tag{\ref{CP}}
\text{compute } Y  = \{ y \in \mathbb{R}^m \; \vert \; \exists x \in \mathbb{R}^n: \; (x, y) \in S  \}.
\end{align}

First, we introduce the properties of boundedness and self-boundedness for the convex projection problem.
\begin{definition}
\label{def_CP_bounded}
The convex projection~\eqref{CP} is called \textbf{bounded} if the set $Y$ is bounded, i.e.~$Y \subseteq B_K$ for some $K > 0$. The convex projection~\eqref{CP} is called \textbf{unbounded} if it is not bounded. The convex projection~\eqref{CP} is called \textbf{self-bounded} if $Y \neq \mathbb{R}^m$ and there exist finitely many points $y^{(1)}, \dots, y^{(k)} \in \mathbb{R}^m$ such that
\begin{align}
\label{eq12}
Y \subseteq \conv \{ y^{(1)}, \dots, y^{(k)} \} + (\cl Y)_{\infty}.
\end{align}
\end{definition}
The reader can justifiably question the differences between Definition~\ref{def_CP_bounded} for the projection problem and Definition~\ref{def_CVOP_bounded} for the vector optimization problem. 
As far as bounded problems are concerned, these differences are intuitively reasonable: Each of the two problems is represented by a set, \eqref{CVOP} by the upper image $\mathcal{G}$ and \eqref{CP} by the set $Y$. The upper image $\mathcal{G}$ is a so-called closed upper set with respect to the ordering cone $C$. Boundedness of the problem~\eqref{CVOP}, therefore, corresponds to boundedness of $\mathcal{G}$ with respect to the ordering cone, in particular as an upper set cannot be topologically bounded. On the other hand, the set $Y$ is not an upper set and it can be bounded in the usual topological sense. Furthermore, as no ordering cone was specified to define the convex projection, boundedness with respect to the ordering cone is not the appropriate concept here.  The difference for the self-bounded problems is less intuitive. We discuss in detail in Section~\ref{sec_self} why for convex projections finitely many points are necessary, while for convex vector optimization problems it is enough to consider one point in the definition of self-boundedness, which provides a motivation for the definition we proposed here. Finally, we will prove in Proposition~\ref{lemma_self-bounded} below that the definitions of boundedness and self-boundedness in Definition~\ref{def_CP_bounded} for the projection problem and Definition~\ref{def_CVOP_bounded} for a particular  multi-objective convex optimization problem correspond one-to-one to each other.

Now, we give a definition of solutions -- both exact and approximate --  for the convex projection. Since we aim towards relating the projection to vector optimization, we only define approximate solutions for the bounded and the self-bounded case. Let a tolerance level $\epsilon > 0$ be fixed.
\begin{definition}
\label{def_CP_sol}
A set $\bar{S} \subseteq S$  is called a \textbf{solution} of~\eqref{CP} if it satisfies 
\begin{align}
\label{eq4}
Y \subseteq \cl \conv \proj_y [\bar{S}].
\end{align}
A non-empty finite set $\bar{S} \subseteq S$  is called a \textbf{finite $\epsilon$-solution} of a bounded problem~\eqref{CP} if
\begin{align}
\label{eq5}
Y \subseteq \conv \proj_y [\bar{S}] + B_{\epsilon}.
\end{align} 
A non-empty finite set $\bar{S} \subseteq S$  is called a  \textbf{finite $\epsilon$-solution} of a self-bounded problem~\eqref{CP} if
\begin{align}
\label{eq13}
Y \subseteq \conv \proj_{y} [\bar{S}] + (\cl Y)_{\infty} + B_{\epsilon}.
\end{align}
\end{definition}

\subsection{An Associated Multi-Objective Convex Problem}
\label{sec_CP_MOCP}
We are interested in examining the connection between convex projections and convex vector optimization. The linear counterpart suggests that multi-objective optimization might be helpful also for solving the convex projection problem.
Inspired by the polyhedral case of~\cite{LW16} and the case of convex bodies in~\cite{SZC18}, we construct a multi-objective problem with the same feasible set $S$ and one additional dimension of the objective space. That is, we consider the following problem
\begin{align}
\label{MOCP}
\min 
\begin{pmatrix}
y \\
- \trans{\mathbf{1}} y
\end{pmatrix} \text{ with respect to } \leq_{\mathbb{R}^{m+1}_+} \text{ subject to } (x, y) \in S,
\end{align}
which is a multi-objective convex optimization problem with a linear objective function and a convex feasible set $S$. To increase readability we add the following notations. The mappings $P: \mathbb{R}^n \times \mathbb{R}^m \rightarrow \mathbb{R}^{m+1}$ and $Q: \mathbb{R}^{m} \rightarrow \mathbb{R}^{m+1}$ are given by the matrices $P = \begin{pmatrix}0 & I \\ 0 & - \trans{\mathbf{1}} \end{pmatrix}$ and $Q = \begin{pmatrix}I \\ - \trans{\mathbf{1}} \end{pmatrix}$, respectively.
The image of the feasible set of~\eqref{MOCP} is given by $P[S] = \{ P (x,y) \; \vert \; (x,y) \in S \}$,  its upper image is denoted by
$$\mathcal{P} = \cl (P[S] + \mathbb{R}^{m+1}_+).$$
The aim of this section is to establish a connection between the convex projection~\eqref{CP} and the multi-objective problem~\eqref{MOCP}, between their properties and their solutions. Then, a solution provided by a solver for problem~\eqref{MOCP} would lead the way to a solution to problem~\eqref{CP}. The following observations are immediate, most importantly the connection between the sets $P[S]$ and $Y$.
\begin{lemma}
\label{lemma_minimizer}
\begin{enumerate}
\item Every feasible point $(x, y) \in S$ is a minimizer of~\eqref{MOCP}.

\item For the sets $Y$ and $P[S]$ it holds $Y = \proj_{-1} [P[S]]$ and $P[S] = Q[Y].$
\end{enumerate}
\end{lemma}
\begin{proof}
The first claim follows from the form of the matrix $P$. Since $Y = \proj_y [S]$, the second claim follows from the matrix equalities $\proj_y = \proj_{-1} \cdot P$ and $P = Q \cdot \proj_y$. 
  \end{proof}
A similar observation was made in~\cite{SZC18} about the problem studied there. Also here, convexity is not used within the proof, so the claim would hold also for non-convex projection problems.

For convenience, we restate from Definition~\ref{def_CVOP_sol} the definition of exact and approximate solutions specifically for problem~\eqref{MOCP} with the direction $ \Vert \mathbb{1} \Vert^{-1} \mathbb{1} \in \interior \mathbb{R}^{m+1}_+$, and applying Lemma~\ref{lemma_minimizer}~(1). A set $\bar{S} \subseteq S$ is a \textbf{solution} of~\eqref{MOCP} if 
\begin{align}
\label{eq1}
\mathcal{P} = \cl \conv (P[\bar{S}] + \mathbb{R}^{m+1}_+).
\end{align}
A nonempty finite set $\bar{S} \subseteq S$ is a \textbf{finite $\epsilon$-solution} of a self-bounded~\eqref{MOCP} if
\begin{align}
\label{eq0}
\mathcal{P} \subseteq \conv P[\bar{S}] + \mathcal{P}_{\infty} - \epsilon\{ \Vert \mathbb{1} \Vert^{-1} \mathbb{1}\}.
\end{align}
This includes also the case of a bounded problem~\eqref{MOCP}, where $\mathcal{P}_{\infty} = \mathbb{R}^{m+1}_+$.

\subsubsection{Relation of Boundedness}
\label{sec_self}
The definition of an approximate solution and the availability of solvers (for~\eqref{CVOP}) depend on whether the problem is (self-)bounded. So before we can relate approximate solutions of the two problems, we need to relate their properties with respect to boundedness. Which also brings us back to the question raised in the previous section -- why does our definition of a self-bounded problem differ between the projection and the vector optimization problem?

An analogy with the corresponding definition in convex vector optimization would suggest that the convex projection~\eqref{CP} should be considered self-bounded if $Y \neq \mathbb{R}^m$ and there exists a single point $y \in \mathbb{R}^m$ such that
\begin{align}
\label{eq7}
Y \subseteq \{y\} + Y_{\infty}.
\end{align}
Instead, we suggested to replace the single point $y$ with a convex polytope and to use the recession cone of the closure of $Y$. This led to the following notion of self-boundedness in Definition~\ref{def_CP_bounded}: The convex projection problem~\eqref{CP} is called self-bounded if $Y \neq \mathbb{R}^m$ and there exist finitely many points $y^{(1)}, \dots, y^{(k)} \in \mathbb{R}^m$ such that
\begin{align}
\tag{\ref{eq12}}
Y \subseteq \conv \{y^{(1)}, \dots, y^{(k)}\} + (\cl Y)_{\infty}.
\end{align}

Two considerations motivate the definition via~\eqref{eq12}: First, the recession cone $Y_{\infty}$ might have an empty interior. Second, the set $Y$ does not need to be closed. Both of these can lead to situations where the multi-objective problem~\eqref{MOCP} is self-bounded, but the projection~\eqref{CP} does not satisfy~\eqref{eq7}. We illustrate this in Examples~\ref{ex_self1} and~\ref{ex_self2}, where for simplicity the trivial projection, i.e. the identity, is used. In both of our examples these issues can be resolved by replacing condition~\eqref{eq7} with condition~\eqref{eq12}. We will prove in Proposition~\ref{lemma_self-bounded} below, that this is also in general the case.

\begin{example}
\label{ex_self1}
Consider the set
\begin{align*}
Y = S = \left\lbrace (y_1, y_2) \; \vert \; y_1^2 + y_2^2 \leq 1  \right\rbrace + \cone \{(0,1)\}
\end{align*}
and its recession cone $Y_{\infty} = \cone \, \{(0,1)\}$. The associated multi-objective problem~\eqref{MOCP} with the upper image
\begin{align*}
\mathcal{P} = \left\lbrace (y_1, y_2, -y_1-y_2) \; \vert \; y_1^2 + y_2^2 \leq 1  \right\rbrace + \cone \{ (0,1,-1), (1,0,0), (0,0,1) \},
\end{align*}
is self-bounded as $\mathcal{P} \subseteq \{(-1,-1,-2)\} + \mathcal{P}_{\infty}$. However, the set $Y$ does not satisfy~\eqref{eq7} for any single point $y \in \mathbb{R}^m$. Because of its empty interior, no shifting of the cone $Y_{\infty}$ can cover the set $Y$ with a non-empty interior. But already two points $y^{(1)} = (-1, -1), \, y^{(2)} = (1, -1)$ suffice for~\eqref{eq12}. 
\end{example}

\begin{example}
\label{ex_self2}
Consider the convex, but not closed, set
\begin{align*}
Y =  S = \left\lbrace (y_1, y_2) \; \vert \; 0 \leq y_1 < 1,  0 \leq y_2  \right\rbrace \cup  \{(1, 0)\}.
\end{align*}
It is not bounded, but its recession cone is trivial, $Y_{\infty} = \{0\}$.
The associated multi-objective problem~\eqref{MOCP} with the upper image
\begin{align*}
\mathcal{P} =  \left\lbrace (y_1, 0, -y_1) \; \vert \; 0 \leq y_1 \leq 1  \right\rbrace + \cone \{ (0,1,-1), (1,0,0), (0,0,1) \},
\end{align*}
is self-bounded as $\mathcal{P} \subseteq \{(0,0,-1)\} + \mathcal{P}_{\infty}$. However, the set $Y$ with its trivial recession cone $Y_{\infty} = \{0\}$ clearly cannot satisfy~\eqref{eq7}. It also wouldn't suffice to just replace the single point $y$ in~\eqref{eq7} by a convex polytope $\conv \{y^{(1)}, \dots, y^{(k)}\}$, since the set $Y$ is not bounded. But the recession cone $(\cl Y)_{\infty} = \cone  \, \{(0,1)\}$ satisfies~\eqref{eq12} with e.g.~the pair of points $y^{(1)} = (0,0), \, y^{(2)} = (1, 0)$. 
\end{example}

The two considerations illustrated here were the motivation for defining self-boundedness via~\eqref{eq12}. The condition~\eqref{eq12} can be considered as a generalization of the definition of self-boundedness found in the literature. For vector optimization problems, Definition~\ref{def_CVOP_bounded} and a definition via~\eqref{eq12} coincide because of two properties of upper images. An upper image is always a closed set and its recession cone is solid as it always contains the ordering cone.

We will now prove that, as long as one uses the more general notion of self-boundedness (Definition~\ref{def_CP_bounded}) for~\eqref{CP},
boundedness, self-boundedness, and unboundedness of the projection~\eqref{CP} are equivalent to boundedness, self-boundedness, and unboundedness of the multi-objective problem~\eqref{MOCP}, respectively. 

\begin{proposition}
\label{lemma_self-bounded}
\begin{enumerate}
\item The convex projection~\eqref{CP} is bounded if and only if the multi-objective problem~\eqref{MOCP} is bounded.
\item The convex projection~\eqref{CP} is self-bounded if and only if the multi-objective problem~\eqref{MOCP} is self-bounded.
\item The convex projection~\eqref{CP} is unbounded if and only if the multi-objective problem~\eqref{MOCP} is unbounded.
\end{enumerate}
\end{proposition}


The proof of Proposition~\ref{lemma_self-bounded} will require the following two lemmas. The first one is a trivial observation and thus stated without proof, but comes in handy as it will also be used later. The second lemma provides a crucial connection between the recession cones of the two problems.

\begin{lemma}
\label{lemma_A2}
Let $A \subseteq \mathbb{R}^n$ be a set and $C \subseteq \mathbb{R}^n$ be a convex cone. Then, $\cl (A + C) = \cl( \cl A + C)$.
\end{lemma}
%

\begin{lemma}
\label{lemma_recession}
For the recession cones of the sets $Y$ and $P[S]$ it holds
\begin{align*}
&(P[S])_{\infty} =  Q [Y_{\infty}] \;\; \text{ and } \;\; Y_{\infty} = \proj_{-1} [(Q[Y])_{\infty}], \\ 
&(\cl P[S])_{\infty} =  Q [(\cl Y)_{\infty}] \;\; \text{ and } \;\; (\cl Y)_{\infty} = \proj_{-1} [(\cl Q[Y])_{\infty}].
\end{align*}
For the recession cone of the upper image it holds
\begin{align}
\label{eq_recession}
\mathcal{P}_{\infty} = Q[(\cl Y)_{\infty}] + \mathbb{R}^{m+1}_+.
\end{align}
\end{lemma}
\begin{proof}
The first four equalities follow from $P[S] = Q[Y]$ (Lemma~\ref{lemma_minimizer}~(2)) and $\cl P[S] = Q [\cl Y]$, the matrix $Q$ being the right-inverse of $\proj_{-1}$ (Lemma~\ref{lemma_minimizer}~(2))  and the fact that $\cl P[S] \subseteq Q [\mathbb{R}^m] = \{ x \in \mathbb{R}^{m+1} \vert \trans{\mathbb{1}} x  =0\}$.

According to Lemma~\ref{lemma_A2}, for the upper image it holds $\mathcal{P} = \cl ( Q [\cl Y] + \mathbb{R}^{m+1}_+)$. The sets $Q [\cl Y]$ and $\mathbb{R}^{m+1}_+$ are both closed and convex. One easily verifies that $Q [(\cl Y)_{\infty}] \subseteq Q [\mathbb{R}^{m}] = \{ x \in \mathbb{R}^{m+1} \vert \trans{\mathbb{1}} x  =0\}$, so $Q [(\cl Y)_{\infty}] \cap -\mathbb{R}^{m+1}_+ = \{0\}$ and the sets $Q [\cl Y]$ and $\mathbb{R}^{m+1}_+$ satisfy all assumptions of Corollary 9.1.2 of~\cite{R70}. Thus, the set $Q [\cl Y] + \mathbb{R}^{m+1}_+$ is closed and $( Q [\cl Y] + \mathbb{R}^{m+1}_+)_{\infty} = Q [(\cl Y)_{\infty}] + \mathbb{R}^{m+1}_+$.
  \end{proof}

\begin{proof}[of Proposition~\ref{lemma_self-bounded}.]
Let us first prove the second claim.
\begin{enumerate} 
\item[$\Rightarrow$] Let~\eqref{CP} be self-bounded. Applying mapping $Q$ onto~\eqref{eq12} and adding the standard ordering cone yields
\begin{align*}
P[S] +  \mathbb{R}^{m+1}_+  \subseteq \conv \{Qy^{(1)}, \dots, Qy^{(k)}\} + Q[(\cl Y)_{\infty}] + \mathbb{R}^{m+1}_+.
\end{align*}
For a vector $q \in \mathbb{R}^{m+1}$, element-wise defined by $q_{i} = \min\limits_{j=1, \dots k} (Qy^{(j)})_{i}$ for $i=1, \dots m+1$, it holds $\conv \{Qy^{(1)}, \dots, Qy^{(k)}\} \subseteq \{q\} + \mathbb{R}^{m+1}_+$, so
\begin{align*}
P[S] +  \mathbb{R}^{m+1}_+  \subseteq \{q\} + Q[(\cl Y)_{\infty}] + \mathbb{R}^{m+1}_+ = \{q\} + \mathcal{P}_{\infty},
\end{align*}
where the last equality is due to~\eqref{eq_recession}. Since the shifted cone $\{q\} + \mathcal{P}_{\infty}$ is closed, self-boundedness of~\eqref{MOCP} follows.

\item[$\Leftarrow$] 
If the problem~\eqref{MOCP} is self-bounded, then there is  a point $q \in \mathbb{R}^{m+1}$ such that
\begin{align} 
\label{eq20}
\mathcal{P}  \subseteq \{q\} + \mathcal{P}_{\infty}.
\end{align}
Since $Q[Y] + \mathbb{R}^{m+1}_+ \subseteq Q[\mathbb{R}^m] + \mathbb{R}^{m+1}_+  = \{ x \in \mathbb{R}^{m+1} \vert \trans{\mathbb{1}} x  \geq 0\}$, for the upper image $\mathcal{P} = \cl ( Q[Y] + \mathbb{R}^{m+1}_+ )$ and its recession cone $\mathcal{P}_{\infty}$ it holds $\mathcal{P}, \mathcal{P}_{\infty} \subseteq \{ x \in \mathbb{R}^{m+1} \vert \trans{\mathbb{1}} x  \geq 0\}$. As the upper image $\mathcal{P}$ contains elements of $Q [\mathbb{R}^m]$, inclusion~\eqref{eq20} is only possible if $\trans{\mathbb{1}}  q \leq 0$.

As $Q[Y] \subseteq \mathcal{P} \cap Q [\mathbb{R}^m]$, from~\eqref{eq20} and~\eqref{eq_recession} it follows 
\begin{align*}
Q[Y]  &\subseteq ( \{q\} + \mathcal{P}_{\infty}) \cap Q [\mathbb{R}^m] = ( \{q\}  + \mathbb{R}^{m+1}_+) \cap Q [\mathbb{R}^m] + Q[(\cl Y)_{\infty}] \\
&= \conv \{ q^{(1)}, \dots, q^{(m+1)} \} + Q[(\cl Y)_{\infty}],
\end{align*}
where $q^{(i)} := q + \vert \trans{\mathbb{1}}  q \vert \cdot e^{(i)}$ for $i = 1, \dots, m+1$.
By applying the projection $\proj_{-1}$  we obtain
\begin{align*}
Y  &\subseteq  \conv \{ \proj_{-1} q^{(1)}, \dots,  \proj_{-1} q^{(m+1)} \} + (\cl Y)_{\infty}.
\end{align*}
\end{enumerate}

The first claim is a consequence of the second claim, relation~\eqref{eq_recession} and the fact that for a bounded~\eqref{CP} it holds $(\cl Y)_{\infty} = \{0\}$ and for a bounded problem~\eqref{MOCP} it holds $\mathcal{P}_{\infty} = \mathbb{R}^{m+1}_+$. The last claim is an equivalent reformulation of the first claim.
  \end{proof}

\subsubsection{Relation of Solutions}
\label{sec_solutions}
Note that under some compactness assumptions the convex projection problem~\eqref{CP} can be approximately solved using the algorithm in~\cite{SZC18}. Here, we are however interested in the connection between problem~\eqref{CP} and the multi-objective problem~\eqref{MOCP} and to see if one can be solved in place of the other. 
Lemma~\ref{lemma_minimizer} already provides us with a close connection between the two problems, which suggests that interchanging the two problems (in some manner) should be possible. We will now examine if, and in what sense, an (exact or approximate) solution of one problem solves the other problem. We start with the exact solutions, where we obtain an equivalence. Note that this result does not need any type of boundedness assumption. Then, we will move towards approximate solutions for the bounded and the self-bounded case.
\begin{theorem}
\label{thm_exact_sol}
A set $\bar{S} \subseteq S$ is a solution of the convex projection~\eqref{CP} if and only if it is a solution of the multi-objective problem~\eqref{MOCP}.
\end{theorem}
\begin{proof}
\begin{enumerate}
\item[$\Rightarrow$] Let $\bar{S} \subseteq S$ be a solution of~\eqref{CP}. Applying the mapping $Q$ onto~\eqref{eq4}, we obtain
\begin{align*}
P[S] \subseteq \cl \conv P[\bar{S}],
\end{align*}
see Lemma~\ref{lemma_minimizer}~(2). Adding the ordering cone, taking the closure and Lemma~\ref{lemma_A2} give
\begin{align*}
 \cl ( P[S] + \mathbb{R}^{m+1}_+ ) \subseteq \cl ( \conv P[\bar{S}] + \mathbb{R}^{m+1}_+ ), 
\end{align*}
therefore, $\bar{S}$ is an infimizer of~\eqref{MOCP}, cf. Definition~\ref{def_CVOP_sol}. According to Lemma~\ref{lemma_minimizer}, it is also a solution.

\item[$\Leftarrow$] Let $\bar{S} \subseteq S$ be a solution of~\eqref{MOCP}. We prove that 
$$P[S] \subseteq \cl \conv P [\bar{S}]$$
 via contradiction: Assume that there is $q \in P[S]$ and $\epsilon > 0$ such that $(\{q\} + B_{\epsilon}) \cap \conv P[\bar{S}] = \emptyset$. According to~\eqref{eq1}, for all $\delta > 0$ it holds $(\{q\} + B_{\delta}) \cap (\conv P[\bar{S}] + \mathbb{R}^{m+1}_+) \neq \emptyset$. That is, there exist $b \in B_{\delta}, \bar{q} \in \conv P[\bar{S}]$ and $r \in \mathbb{R}^{m+1}_+$ such that $q + b = \bar{q} + r$. 
As $b_i \leq \delta$ one has $\trans{\mathbb{1}}b \leq (m+1) \delta$. Thus, $\| b - r \| \leq \| b \| + \|-r\| \leq \|b\| + \trans{\mathbb{1}}r =\|b\| + \trans{\mathbb{1}}b \leq \delta + (m+1)\delta$, where $\trans{\mathbb{1}}r =\trans{\mathbb{1}}b$ follows from $q + b = \bar{q} + r$ and $\trans{\mathbb{1}} q = \trans{\mathbb{1}} \bar{q} = 0$. Hence, one obtains $b-r \in B_{\delta \cdot (m+2)}$, so for the choice of $\delta := \epsilon / (m+2)$ we get a contradiction. 
 
Finally, applying $\proj_{-1}$ gives $Y \subseteq \cl \conv \proj_y [\bar{S}]$ and $\bar{S}$ is a solution of~\eqref{CP}.
\end{enumerate}
  \end{proof}

This result is analogous to Theorem 3 of~\cite{LW16} for the polyhedral case. Here it is unnecessary to prove the existence of a solution when the problem is feasible as the full feasible set is a solution according to Definitions~\ref{def_CVOP_sol} and~\ref{def_CP_sol} (consider Lemma~\ref{lemma_minimizer}). Since (exact) solutions are usually unattainable in practice, it is of greater interest to see how the approximate solutions relate. 
 It turns out, the approximate solutions are equivalent only up to an increased error. We will see that in order to use an approximate solution of one problem to solve the other problem, we need to  increase the tolerance proportionally to a certain multiplier. These multipliers depend both on the dimension $m$ of the problem and on the $p$-norm under consideration. Recall that the definition of a finite $\epsilon$-solution of either of the two problems depends implicitly on the fixed $p \in [1, \infty]$. 
For this reason, we denote the $p$ explicitly in this section.

We deal with the bounded and the self-bounded case jointly. As we mentioned before, a bounded~\eqref{CP} has a trivial recession cone $(\cl Y)_{\infty} = \{0\}$ and if~\eqref{MOCP} is bounded it holds $\mathcal{P}_{\infty} = \mathbb{R}^{m+1}_+$. This verifies that~\eqref{eq13} and~\eqref{eq0} correcly define a finite $\epsilon$-solution of a bounded~\eqref{CP} and  a bounded problem~\eqref{MOCP}, respectively. Proposition~\ref{lemma_self-bounded} provided an equivalence between the self-boundedness of the two problems. The following two theorems contain the main results.

\begin{theorem}
\label{thm_CPsolMOCP}
Let the convex projection~\eqref{CP} be self-bounded and fix $p \in [1, \infty]$. 
If $\bar{S} \subseteq S$ is a finite $\epsilon$-solution of~\eqref{CP} (under the $p$-norm), then it is  a finite $\left(\underline{\kappa} \cdot \epsilon \right)$-solution of~\eqref{MOCP} (under the $p$-norm), where $\underline{\kappa} = \underline{\kappa} (m, p) = m^{\frac{p-1}{p}} \cdot (m+1)^{\frac{1}{p}}$ for $p \in [1, \infty)$ and $\underline{\kappa} = \underline{\kappa} (m, \infty ) = m$ for $p = \infty$.
\end{theorem}
\begin{proof}
Multiplying~\eqref{eq13} with $Q$ and adding the ordering cone $\mathbb{R}^{m+1}_+$ yields
\begin{align}
\label{eq10}
P[S] +  \mathbb{R}^{m+1}_+  \subseteq \conv P[\bar{S}] + Q[(\cl Y)_{\infty}] + \mathbb{R}^{m+1}_+ + Q \left[ B_{\epsilon}^{p} \right].
\end{align}
Let $p \in [1, \infty)$. 
For $b \in B_{\epsilon}^{p}$ it holds $\vert b_i \vert \leq \epsilon$ for each $i = 1, \dots, m$. By considering the problem $\max \sum_{i=1}^m b_i \text{ s.t. } \sum_{i=1}^m b_i^p \leq 1$, we also obtain $\vert \sum_{i=1}^m b_i \vert  \leq m^{\frac{p-1}{p}} \cdot \epsilon$. Thus, we have $- \left( m^{\frac{p-1}{p}} \cdot \epsilon \right) \mathbb{1} \leq Qb$ for all $b \in B_{\epsilon}^{p}$, which leads to
\begin{align}
\label{eqN02}
Q \left[ B_{\epsilon}^{p} \right] \subseteq - \left( m^{\frac{p-1}{p}} \cdot \Vert \mathbb{1} \Vert_p \cdot \epsilon \right) \left\lbrace \Vert \mathbb{1} \Vert_p^{-1} \mathbb{1}    \right\rbrace + \mathbb{R}^{m+1}_+.
\end{align}
From the equations~\eqref{eq_recession},~\eqref{eq10} and~\eqref{eqN02} and $\Vert \mathbb{1} \Vert_p = (m+1)^{\frac{1}{p}}$ it follows
\begin{align*}
P[S] +  \mathbb{R}^{m+1}_+  &\subseteq 
 \conv P[\bar{S}] + \mathcal{P}_{\infty} - \left( m^{\frac{p-1}{p}} \cdot (m+1)^{\frac{1}{p}} \cdot \epsilon \right) \left\lbrace \Vert \mathbb{1} \Vert_p^{-1} \mathbb{1}    \right\rbrace.
\end{align*}

Now consider $p = \infty$. It similarly holds $\vert b_i \vert \leq \epsilon, i = 1, \dots, m$ , as well as $\vert \sum_{i=1}^m b_i \vert  \leq m \cdot \epsilon$ for $b \in B_{\epsilon}^{\infty}$. Therefore, the same steps yield $Q \left[ B_{\epsilon}^{\infty} \right] \subseteq - (m \cdot \epsilon ) \left\lbrace \mathbb{1}  \right\rbrace + \mathbb{R}^{m+1}_+$ and $P[S] +  \mathbb{R}^{m+1}_+  \subseteq \conv P[\bar{S}] + \mathcal{P}_{\infty} - \left( m \cdot \epsilon \right) \left\lbrace \Vert \mathbb{1} \Vert_p^{-1} \mathbb{1}    \right\rbrace$.

Finally, the sets $\conv P[\bar{S}]$ and $\mathcal{P}_{\infty}$ satisfy the assumptions of Corollary 9.1.2 of~\cite{R70} -- both sets are closed convex and $\conv P[\bar{S}]$ has a trivial recession cone as it is bounded -- therefore, the set $\conv P[\bar{S}] + \mathcal{P}_{\infty}$ is closed. Since the shifted set is also closed, we obtain $\mathcal{P} \subseteq  \conv P[\bar{S}] + \mathcal{P}_{\infty} - \left(\underline{\kappa} \cdot \epsilon \right) \left\lbrace \Vert \mathbb{1} \Vert_p^{-1} \mathbb{1}    \right\rbrace$ and $\bar{S}$ is a finite $\left(\underline{\kappa} \cdot \epsilon \right)$-solution of~\eqref{MOCP} according to Lemma~\ref{lemma_minimizer}.
  \end{proof}

\begin{theorem}
\label{thm_MOCPsolCP}
Let the multi-objective problem~\eqref{MOCP} be self-bounded and fix $p \in [1, \infty]$. 
If $\bar{S} \subseteq S$ is a finite $\epsilon$-solution of~\eqref{MOCP} (under the $p$-norm), then it is a finite $\left( \overline{\kappa} \cdot \epsilon \right)$-solution of~\eqref{CP} (under the $p$-norm), where $\overline{\kappa} = \overline{\kappa} (m, p) = \left(\frac{m^p + m - 1}{m+1} \right)^{\frac{1}{p}}$ for $p \in [1, \infty)$ and $\overline{\kappa} = \overline{\kappa} (m, \infty ) = m$ for $p = \infty$.
\end{theorem}

\begin{proof}
We give the proof for $p\in [1, \infty )$. For $p = \infty$ the same steps work, only the expressions $(m^p + m - 1)^{\frac{1}{p}}$ and $\left( \frac{m^p + m - 1}{m+1} \right)^{\frac{1}{p}}$ need to be replaced by their value in the limit, $m$.
For the set $Q[Y] = P[S] \subseteq \mathcal{P} \cap Q [\mathbb{R}^m]$, according to~\eqref{eq0} and~\eqref{eq_recession}, we obtain
\begin{align*}
Q[Y] &\subseteq \left( \conv P[\bar{S}] + \mathcal{P}_{\infty} - \epsilon \left\lbrace \Vert \mathbb{1} \Vert_p^{-1} \mathbb{1} \right\rbrace \right) \cap Q [\mathbb{R}^m] \\
 &= \conv P[\bar{S}] + Q[(\cl Y)_{\infty}] + \left( \mathbb{R}^{m+1}_+ - \epsilon \left\lbrace \Vert \mathbb{1} \Vert_p^{-1} \mathbb{1} \right\rbrace \right) \cap Q [\mathbb{R}^m].
\end{align*}
A projection onto the first $m$ coordinates yields
\begin{align*}
Y  \subseteq \conv \proj_y [\bar{S}] + (\cl Y)_{\infty} + \proj_{-1} \left[ \left( \mathbb{R}^{m+1}_+ - \epsilon \left\lbrace \Vert \mathbb{1} \Vert_p^{-1} \mathbb{1} \right\rbrace \right) \cap Q [\mathbb{R}^m] \right].
\end{align*}
What remains is to show that the last set on the right-hand side is contained in an error-ball
\begin{align}
\label{eq_E06}
\proj_{-1} \left[ \left( \mathbb{R}^{m+1}_+ - \epsilon \left\lbrace \Vert \mathbb{1} \Vert_p^{-1} \mathbb{1} \right\rbrace \right) \cap Q [\mathbb{R}^m] \right]  
\subseteq 
B_{ \left( \frac{m^p + m - 1}{m+1} \right)^{\frac{1}{p}} \cdot \epsilon}^p,
\end{align} 
which would finish the proof.

Now, to see that~\eqref{eq_E06} holds true, note that by $Q [\mathbb{R}^m] = \{ x \in \mathbb{R}^{m+1} \vert \trans{\mathbb{1}} x  =0\}$ one has
\begin{align*}
\left( \mathbb{R}^{m+1}_+ - \epsilon \left\lbrace \Vert \mathbb{1} \Vert_p^{-1} \mathbb{1} \right\rbrace \right) \cap Q [\mathbb{R}^m] 
&= \epsilon \Vert \mathbb{1} \Vert_p^{-1} \cdot \left\lbrace r - \mathbb{1} \; \vert \; r \in \mathbb{R}^{m+1}_+, \trans{\mathbb{1}} (r - \mathbb{1}) = 0  \right\rbrace.
\end{align*}
A projection onto the first $m$ coordinates yields
\begin{align*}
\proj_{-1} &\left[ \left( \mathbb{R}^{m+1}_+ - \epsilon \left\lbrace \Vert \mathbb{1} \Vert_p^{-1} \mathbb{1} \right\rbrace \right) \cap Q [\mathbb{R}^m] \right] 
= \epsilon \cdot \Vert \mathbb{1} \Vert_p^{-1} \cdot \left\lbrace r - \mathbf{1} \in \mathbb{R}^m \; \vert \; r \in \mathbb{R}^m_+, \trans{\mathbf{1}} r \leq m+1  \right\rbrace  .
\end{align*}
The optimization problem
\begin{align}
\label{op}
\max \| r - \mathbf{1} \|_p \; \text{ s.t. } \; r \in \mathbb{R}^{m}_+, \trans{\mathbf{1}} r \leq m+1 
\end{align}
maximizes a convex objective over a compact polyhedron. As such, the maximum needs to be attained in one of the vertices of the polyhedron. Therefore, the optimal objective value of problem~\eqref{op} is $\| (m+1) \cdot e^{(1)} - \mathbf{1} \|_p = (m^p + m - 1)^{\frac{1}{p}}$, which shows $$\left\lbrace r - \mathbf{1} \in \mathbb{R}^m \; \vert \; r \in \mathbb{R}^m_+, \trans{\mathbf{1}} r \leq m+1  \right\rbrace \subseteq B_{(m^p + m - 1)^{\frac{1}{p}}}^{p}.$$ Multiplication with the constant $\epsilon \cdot \Vert \mathbb{1} \Vert_p^{-1}$ yields~\eqref{eq_E06}.

  \end{proof}

Both of the above theorems involve increasing the error tolerance proportionally to a certain multiplier. The multipliers $\underline{\kappa}$ and $\overline{\kappa}$ both depend on the dimension $m$ and on the $p$-norm. Of those two the dimension plays a more important role. One can verify that for all $p \in [1, \infty]$ the value of $\underline{\kappa} (m, p)$ lies between $m$ and $m+1$ and the value of $\overline{\kappa} (m, p)$ is bounded from above by $m$. The multipliers in the maximum norm correspond to the limits, i.e.~it holds $\underline{\kappa} (m, \infty ) = \lim\limits_{p \to \infty} \underline{\kappa} (m, p)$ and $\overline{\kappa} (m, \infty ) = \lim\limits_{p \to \infty} \overline{\kappa} (m, p)$. We list the values of the two multipliers for the three most popular norms, the Manhattan, the Euclidean and the maximum norm, in Table~\ref{tab1}.

\begin{table}[h]
\center
\caption{\label{tab1}  Multipliers $\underline{\kappa} (m, p)$ and $\overline{\kappa} (m, p)$ for  $p = 1, 2$ and $\infty$. }
\begin{tabular}{c c c  c c c c}
\hline \hline
 & p &														& $\;\;\;\;$ & $1$					& $2$						& $\infty$ \\ \hline
$\underline{\kappa}$ & $=$ & $m^{\frac{p-1}{p}} \cdot (m+1)^{\frac{1}{p}}$	&& $m+1$		& $\sqrt{m (m+1)}$	& $m$ \\
$\overline{\kappa}$ & $=$ & $\left( \frac{m^p + m - 1}{m+1} \right)^{\frac{1}{p}}$	&& $\frac{2m-1}{m+1}$	& $\sqrt{\frac{m^2+m-1}{m+1}}$	& $m$ \\ \hline \hline
\end{tabular}
\end{table}

Theorem~\ref{thm_CPsolMOCP} (respectively Theorem~\ref{thm_MOCPsolCP}) guarantees that increasing the tolerance $\underline{\kappa}$-fold (respectively $\overline{\kappa}$-fold) is sufficient. We could, however, ask if some smaller increase of the tolerance might not suffice instead. In Subsection~\ref{subsec_ex} below, we provide examples where the multipliers $\underline{\kappa}$ and $\overline{\kappa}$ are attained. Therefore, the results of Theorems~\ref{thm_CPsolMOCP} and~\ref{thm_MOCPsolCP} cannot be improved.

A finite $\epsilon$-solution of a (self-)bounded~\eqref{CVOP} depends not only on the tolerance $\epsilon$ and the underlying $p$-norm, but also on the direction $c \in \interior C$ used within the defining relation~\eqref{eq16}, respectively~\eqref{eq17}. As we stated at the beginning of this section, for the multi-objective problem~\eqref{MOCP} we use the direction $ \Vert \mathbb{1} \Vert^{-1} \mathbb{1} \in \interior \mathbb{R}^{m+1}_+$. Since this direction appears within the proofs of Theorems~\ref{thm_CPsolMOCP} and~\ref{thm_MOCPsolCP}, how much do our results depend on it? The structure of the results would remain unchanged regardless of the considered direction, only the multipliers $\underline{\kappa}$ and $\overline{\kappa}$ are direction-specific -- variants of relations~\eqref{eqN02} and~\eqref{eq_E06} hold for arbitrary fixed direction $r \in \interior \mathbb{R}^{m+1}_+$ with appropriately adjusted multipliers.

The last question that remains open is whether approximate solutions exist. In the following lemma we prove their existence for self-bounded convex projections.

\begin{lemma}
\label{lemma_sol_exist}
There exists a finite $\epsilon$-solution to a self-bounded convex projection~\eqref{CP} for any $\epsilon > 0$.
\end{lemma}
\begin{proof}
This proof is based on Proposition 3.7 of~\cite{U18}. We use it to prove the existence of a finite $\frac{\epsilon}{\underline{\kappa}}$-solution of the associated problem~\eqref{MOCP} and then apply Theorem~\ref{thm_MOCPsolCP}. Fix $\epsilon > 0$ and set $\xi := \frac{\epsilon}{2\underline{\kappa}}$ and $\delta := \frac{\epsilon}{2 \underline{\kappa} \| \mathbb{1} \| }$. Since problem~\eqref{MOCP} is self-bounded, the upper image $\mathcal{P}$ and its recession cone $\mathcal{P}_{\infty}$ satisfy the assumptions of Proposition 3.7 of~\cite{U18}. Then, for a tolerance $\xi > 0$ there exists a finite set of points $\bar{A} \subseteq \mathcal{P}$ such that
\begin{align} 
\label{eq_E10}
\mathcal{P} \subseteq \conv \bar{A} + \mathcal{P}_{\infty} - \xi \{ \|\mathbb{1}\|^{-1} \mathbb{1} \}.
\end{align}
Since $\mathcal{P} = \cl (P[S] + \mathbb{R}^{m+1}_+)$, for each $\bar{a} \in \bar{A}$ there exists $(\bar{x}, \bar{y}) \in S$ such that $\bar{a} \in P(\bar{x}, \bar{y})  + \mathbb{R}^{m+1}_+ + B_{\delta}$. Denote by $\bar{S} \subseteq S$ the (finite) collection of such feasible points $(\bar{x}, \bar{y}) \in S$, one per each element of $\bar{A}$. We obtain
\begin{align}
\label{eq_E11}
\bar{A} \subseteq P[\bar{S}] + \mathbb{R}^{m+1}_+ + B_{\delta} \subseteq P[\bar{S}] + \mathbb{R}^{m+1}_+ - (\delta \|\mathbb{1}\|) \{ \|\mathbb{1}\|^{-1} \mathbb{1} \}.
\end{align}
Equations~\eqref{eq_E10} and~\eqref{eq_E11} jointly give
\begin{align*}
\mathcal{P} \subseteq \conv P[\bar{S}]  + \mathcal{P}_{\infty} + \mathbb{R}^{m+1}_+ - \left( \xi + \delta \|\mathbb{1}\| \right) \{ \|\mathbb{1}\|^{-1} \mathbb{1} \} = \conv P[\bar{S}]  + \mathcal{P}_{\infty} -  \frac{\epsilon}{\underline{\kappa}} \{ \|\mathbb{1}\|^{-1} \mathbb{1} \},
\end{align*}
which shows that $\bar{S}$ is a finite $\frac{\epsilon}{\underline{\kappa}}$-solution of~\eqref{MOCP}. According to Theorem~\ref{thm_MOCPsolCP}, $\bar{S}$ is a finite $\epsilon$-solution of~\eqref{CP}.
  \end{proof}
Lemma~\ref{lemma_sol_exist} also proves the existence of a finite $\epsilon$-solution of a self-bounded problem~\eqref{MOCP}. Unlike Proposition 4.3 of~\cite{LRU14}, this result does not require a compact feasible set. In practice, however, we are restricted by the assumptions of the solvers.

\subsection{Examples}
\label{subsec_ex}
We start with three theoretical examples, one for each theorem of the previous subsection, before providing two numerical examples. For simplicity, Examples~\ref{ex_bounded1} and~\ref{ex_bounded2} contain a trivial projection, i.e.~the  identity. One could modify the feasible sets to include additional dimensions, but since these examples are intended to illustrate theoretical properties, we chose to keep them as simple as possible.

 First, we look at exact solutions for which equivalence was proven in Theorem~\ref{thm_exact_sol}. Example~\ref{ex_solution} illustrates that a solution of~\eqref{MOCP} solves~\eqref{CP} only 'up to the closure', i.e.~the closure in~\eqref{eq4} mirrors the closure in~\eqref{eq1}. 
\begin{example}
\label{ex_solution}
Consider the feasible set $S = \{ (x, y) \; \vert \; x^2 + y^2 \leq 1 \}$, which projects onto $Y = \{ y \; \vert \; x^2 + y^2 \leq 1 \} = [-1, 1]$. 
Correspondingly, $P[S] = \{y \cdot (1, -1) \; \vert \; y \in [-1, 1] \}$. 
The set 
$$\bar{S} = \{ (x,y) \; \vert \; x^2 + y^2 = 1 \} \backslash \{(0, -1)\}$$
is a solution of the multi-objective problem. However, projected onto its $y$-element it gives $\proj_y [\bar{S}] = (-1, 1]$. Therefore, $\bar{S}$ is a solution of the projection but $Y \neq \conv \proj_y [\bar{S}]$.
\end{example}

Second, we give an example where the multiplier $\underline{\kappa}$ introduced in Theorem~\ref{thm_CPsolMOCP} is attained. Namely, we provide a convex projection and its finite $\epsilon$-solution $\bar{S}$. This $\bar{S}$ is not a finite $\bar{\epsilon}$-solution of the associated multi-objective problem for any $\bar{\epsilon} < \underline{\kappa} \epsilon$. This shows that Theorem~\ref{thm_CPsolMOCP} would not hold with any smaller multiplier. The example holds for all dimensions $m$ and for all $p$-norms, so $p \in [1, \infty]$ is explicitly denoted.
\begin{example}
\label{ex_bounded1}
Let us start with an example in dimension $m=2$.
Consider the approximate solution $\bar{S} = \{ (0,0), (1, 0), (0, 1)\}$ of the projection
$$Y = S = \{ (y_1, y_2) : y_1^2 + y_2^2 \leq 1, y_1 \geq 0, y_2 \geq 0 \}.$$
The smallest tolerance for which the set $\bar{S}$ is a finite $\epsilon^{(p)}$-solution of this convex projection (under the $p$-norm) is  $\epsilon^{(p)} = \| (\frac{1}{\sqrt{2}}, \frac{1}{\sqrt{2}}) - (\frac{1}{2}, \frac{1}{2}) \|_p = \frac{\sqrt{2}-1}{2} \cdot 2^{\frac{1}{p}}$. 
Now consider the associated multi-objective problem.
To cover the vector $\left( \frac{1}{\sqrt{2}}, \frac{1}{\sqrt{2}}, -\sqrt{2} \right) \in P[S]$ by the set $\conv P[\bar{S}] - x \{\mathbb{1}\} + \mathbb{R}^3_+$ we need $x \geq \sqrt{2} - 1$. Considering normalization by $\Vert \mathbb{1} \Vert_p = 3^{\frac{1}{p}}$, the set $\bar{S}$ is a finite $\bar{\epsilon}^{(p)}$-solution of the multi-objective problem (under the $p$-norm) only if $\bar{\epsilon}^{(p)} \geq (\sqrt{2}-1) \cdot 3^{\frac{1}{p}} = \underline{\kappa} \cdot \epsilon^{(p)}$, where $\underline{\kappa} = 2^{\frac{p-1}{p}}\cdot 3^{\frac{1}{p}}$.

For a self-bounded example consider $Y = S = \{ (y_1, y_2) : y_1^2 + y_2^2 \leq 1, y_1 \geq 0, y_2 \geq 0 \} + \cone \{(0, -1)\}$ with the same approximate solution.

The same structure yields an example also in the $m$-dimensional space for any $m \geq 2$. Consider the approximate solution $\bar{S} = \{ \mathbb{0}, e^{(1)}, \dots, e^{(m)}\}$ of the projection
$$Y = S = \left\lbrace y \in \mathbb{R}^m_+ : \sum_{i=1}^m y_i^2 \leq 1 \right\rbrace.$$
The smallest tolerance for which $\bar{S}$ is a finite $\epsilon^{(p)}$-solution of this convex projection (under the $p$-norm) is $\epsilon^{(p)} = \| \frac{1}{\sqrt{m}} \mathbb{1} - \frac{1}{m} \mathbb{1} \| = \frac{\sqrt{m} - 1}{m} m^{\frac{1}{p}}.$ The upper image of the associated multi-objective problem contains the vector $\left( \frac{1}{\sqrt{m}}, \dots, \frac{1}{\sqrt{m}}, -\sqrt{m} \right)$. Therefore, the set $\bar{S}$ is a finite $\bar{\epsilon}^{(p)}$-solution of the multi-objective problem (under the $p$-norm) only if $\bar{\epsilon}^{(p)} \geq (\sqrt{m}-1) \cdot (m+1)^{\frac{1}{p}}$, where $(\sqrt{m}-1) \cdot (m+1)^{\frac{1}{p}} = \underline{\kappa} \cdot \epsilon^{(p)}$.
\end{example}

Third, we give an example where the multiplier $\overline{\kappa}$ introduced in Theorem~\ref{thm_MOCPsolCP} is attained. Here a multi-objective problem (associated to a given projection) and its finite $\epsilon$-solution $\bar{S}$ are given. We show that this $\bar{S}$ is a finite $\bar{\epsilon}$-solution of the projection only for $\bar{\epsilon} \geq \overline{\kappa} \epsilon$. This shows that Theorem~\ref{thm_MOCPsolCP} would not hold with a smaller multiplier. The $p$-norm is again explicitly denoted.
\begin{example}
\label{ex_bounded2}
Let us start with an example in dimension $m=2$.
Consider the projection $Y = S = \conv \{ (0,0), (2, -1) \}$ with the associated upper image $\mathcal{P} = \conv \{ (0,0,0), (2, -1, -1) \} + \mathbb{R}^3_+$. 
The smallest tolerance for which the set $\bar{S} = \{ (0,0) \}$ is a finite $\epsilon^{(p)}$-solution of the multi-objective problem (under the $p$-norm) is $\epsilon^{(p)} = \Vert \mathbb{1} \Vert_p = 3^{\frac{1}{p}}$. On the other hand, the set $\bar{S}$ is a finite $\bar{\epsilon}^{(p)}$-solution of the convex projection (under the $p$-norm) only for a tolerance $\bar{\epsilon}^{(p)} \geq \Vert (2, -1) \Vert_p = (2^p + 1)^{\frac{1}{p}} = \overline{\kappa} \cdot \epsilon^{(p)}$, where $\overline{\kappa} = \left( \frac{2^p + 1}{3} \right)^{\frac{1}{p}}$.

For a self-bounded example consider $Y = S = \conv \{ (0,0), (2, -1) \} + \cone \{(-1, 0)\}$ with the same approximate solution.

This example can be extended to any dimension $m \geq 2$. Consider the projection $Y = S = \conv \{ \mathbb{0}, (m, -1, \dots, -1) \} \subseteq \mathbb{R}^m$, the associated multi-objective problem with upper image $\mathcal{P} = \conv \{ \mathbb{0}, (m, -1, \dots, -1) \} + \mathbb{R}^{m+1}_+$, and the approximate solution $\bar{S} = \{ \mathbb{0}\}$. The set $\bar{S}$ is a finite $\epsilon^{(p)}$-solution of the multi-objective problem (under the $p$-norm) for $\epsilon^{(p)} = \| \mathbb{1} \|_p = (m+1)^\frac{1}{p}$. However, it is a finite $\bar{\epsilon}^{(p)}$-solution of the convex projection (under the $p$-norm) only for a tolerance $\bar{\epsilon}^{(p)} \geq \| (m, -1, \dots, -1) \|_p = (m^p + m -1)^\frac{1}{p}$, where $(m^p + m -1)^\frac{1}{p} = \overline{\kappa} \epsilon^{(p)}$.
\end{example}

Finally, we provide two numerical examples. 
As both have a compact feasible set, one could apply also the algorithm in~\cite{SZC18} to approximately solve the convex projection problem directly. As here, our motivation is to 
illustrate the theory of the last subsection, we will use in both cases the algorithm of~\cite{LRU14} for convex vector optimization problems to solve the convex projection problem by applying Theorem~\ref{thm_MOCPsolCP}.
\begin{example}
Consider the feasible set
\begin{align*}
S = \left\lbrace (y_1, y_2, x) \; \vert \; (y_1 - 1)^2 + \frac{(y_2 - 2)^2}{2^2} + (x - 1)^2 \leq 1, \right. \;\; \\
\left. \frac{(y_1-2)^2}{2^2} + (y_2 - 1)^2 + \frac{(x - 2)^2}{2^2} \leq 1 \right\rbrace
\end{align*}
consisting of an intersection of two three-dimensional ellipses. We project this set onto the first two coordinates, so we are interested in the set $Y = \{ (y_1, y_2) \; \vert \; (y_1, y_2, x) \in S \}$. We used the associated multi-objective problem to compute an approximation of $Y$. The algorithm of~\cite{LRU14} was used for the numerical computations. In Figure~\ref{fig1} we depict the (inner) approximations obtained for various error tolerances of the algorithm.
\end{example}
\begin{figure*}[h]
\includegraphics[width = \textwidth]{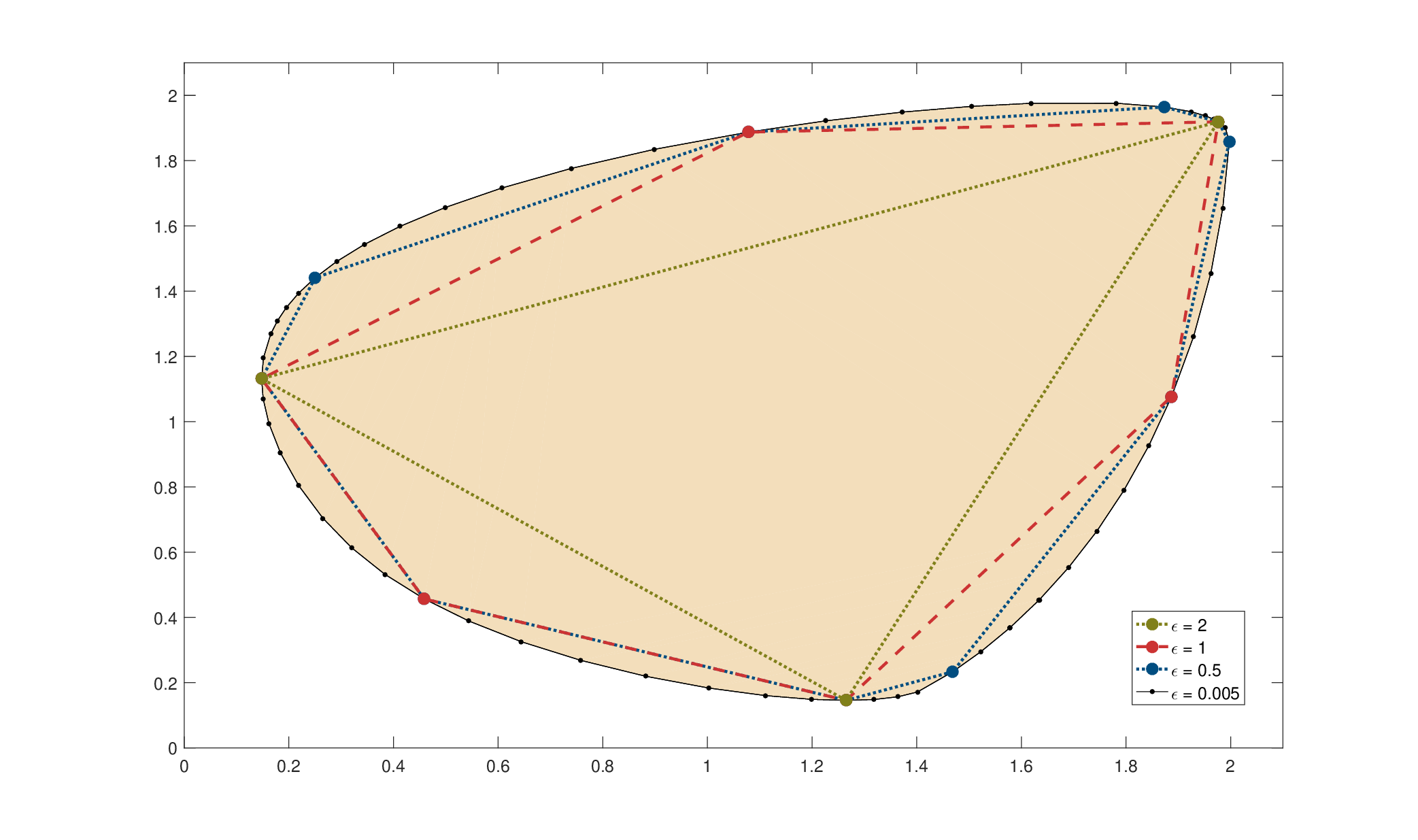}
\caption{\label{fig1}Intersection of two three-dimentional ellipses projected on the first two coordinates. Approximations obtained for various error tolerances of the associated multi-objective problem.}
\end{figure*}

\begin{example}
We add one more dimension. This means, we are computing a three-dimensional set $Y = \left\lbrace (y_1, y_2, y_3) \; \vert \; (y_1, y_2, y_3, x) \in S \right\rbrace$ obtained by projecting an intersection of two four-dimensional ellipses
\begin{align*}
S = \left\lbrace (y_1, y_2, y_3,  x) \; \vert \; (y_1 - 1)^2 + \frac{(y_2 - 2)^2}{2^2} + (y_3 - 1)^2 + \frac{(x-2)^2}{2^2} \leq 1, \right. \;\;\; \\
\left. \frac{(y_1-2)^2}{2^2} + (y_2 - 1)^2 + \frac{(y_3 - 2)^2}{2^2} + (x-1)^2 \leq 1 \right\rbrace
\end{align*}
onto the first three coordinates. Figure~\ref{fig2} contains an (inner) approximation of the set $Y$ obtained by solving the associated multi-objective problem with an error tolerance $\epsilon = 0. 01$. Once again the algorithm of~\cite{LRU14} was used.
\end{example}
\begin{figure*}[h]
\includegraphics[width = \textwidth]{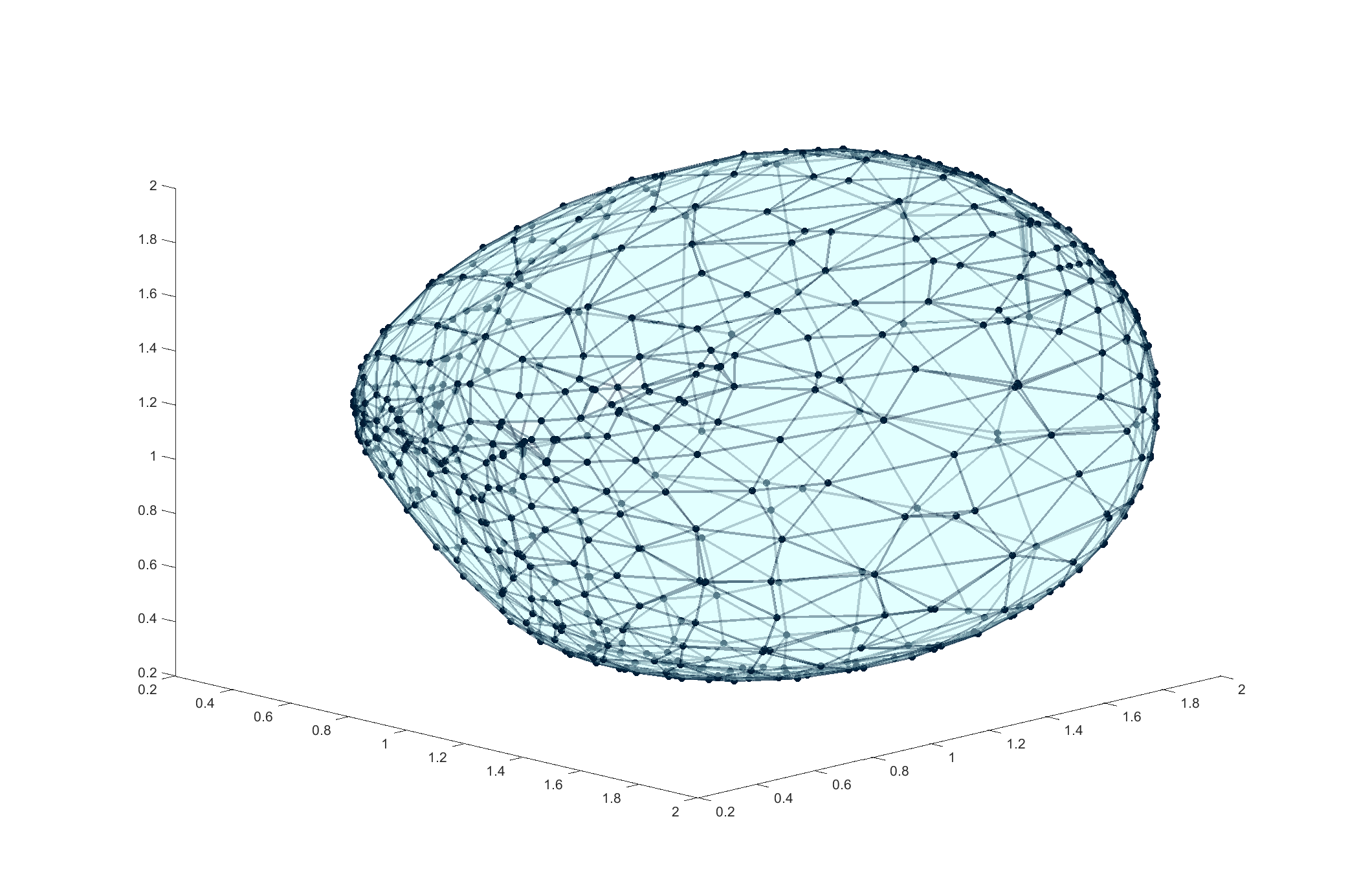}
\caption{\label{fig2}Intersection of two four-dimensional ellipses projected on the first three coordinates. Approximation obtained for a tolerance $\epsilon = 0.01$.}
\end{figure*}

\section{Convex Projection Corresponding to~\eqref{CVOP}}
\label{sec_CVOP_CP}
In the polyhedral case it is also possible to construct a projection associated to any given vector linear program.
In~\cite{LW16}, it is shown that if a solution of the vector linear program exists, it can be obtained from a solution of the associated projection. This, together with the connection between polyhedral projection and its associated multi-objective problem, led to an equivalence between polyhedral projection, multi-objective linear programming and vector linear programming in~\cite{LW16}. This also allows to construct a multi-objective linear program corresponding to any given vector linear program, where the dimension of the objective space is increased only by one.

We will now investigate this in the convex case. 
We start with a general convex vector optimization problem and construct an associated convex projection problem, once again taking inspiration from the polyhedral case in~\cite{LW16}. Analogously to the previous section, we investigate the connection between the two problems, their properties and their solution. While the connection does exist, two problems occur, which show that the equivalence obtained in the polyhedral case cannot be generalized to the convex case in full extent.
Firstly, the associated convex projection is never a bounded problem. So even if a bounded~\eqref{CVOP} is given, its associated convex projection (and thus, its associated multi-objective convex optimization problem) is just self-bounded. Recall that while solvers for a bounded~\eqref{CVOP} are available, there is not yet a solver for self-bounded problems. Secondly, and even more severely, the convex projection provides only (exact or approximate) infimizers of the convex vector optimization problem, but not (exact or approximate) solutions. For exact solutions one can provide conditions to resolve this issue (see Lemma~\ref{lemma_construct_sol} below). This is, however, not possible for the in practice more important approximate solutions.

We now deduce these results in detail.
Recall from Section~\ref{sec_CVOP} the convex vector optimization problem
\begin{align*}
\tag{CVOP}
\begin{split}
\min \Gamma (x) \text{ with respect to } \leq_C \text{ subject to } x \in \mathcal{X}
\end{split}
\end{align*}
with its feasible set $\mathcal{X}$ and its upper image $\mathcal{G} = \cl (\Gamma[\mathcal{X}] + C).$
To obtain an associated projection, we define the set
$S_a = \{ (x, y) \; \vert \; x \in \mathcal{X},   y \in \{\Gamma (x)\} + C\}$, again motivated by~\cite{LW16}. The projection problem with this feasible set is
\begin{align}
\label{Pv}
\text{compute } Y_a = \proj_y [ S_a ] = \{ y  \; \vert \; \exists x \in \mathcal{X}, \Gamma (x) \leq_C y \}.
\end{align}
Clearly,~\eqref{Pv} is feasible if and only if~\eqref{CVOP} is feasible. Note that if the feasible set $\mathcal{X}$ is given via a collection of inequalities, then the new feasible set $S_a$ can be expressed similarly. The following lemma provides convexity and establishes the connection between~\eqref{CVOP} and~\eqref{Pv}.

\begin{lemma}
\label{lemmaYa}
The set $S_a$ is convex and, therefore, the problem~\eqref{Pv} is a convex projection.
Additionally, $Y_a = \Gamma [\mathcal{X}] + C$ and $\mathcal{G} = \cl Y_a$.
\end{lemma}
\begin{proof}
Convexity of $S_a$ follows from the convexity of $\mathcal{X}$, the $C$-convexity of $\Gamma$ and the convexity of $C$.
For the rest consider $Y_a = \{ y \; \vert \; \exists x \in \mathcal{X},  y \in \{\Gamma (x)\} + C\} = \Gamma [\mathcal{X}] + C.$
  \end{proof}

Given this close connection between the sets $\mathcal{G}$ and $Y_a$, we expect a similar connection to exist between the properties and the solutions concepts of~\eqref{CVOP} and~\eqref{Pv}. Since the set $Y_a$ contains the (shifted) ordering cone $C$, the projection problem~\eqref{Pv} cannot be bounded. But still, a bounded problem~\eqref{CVOP} has (by Lemma~\ref{lemma_self-bd}) its counterpart in the properties of the (self-bounded) associated projection~\eqref{Pv}, namely $(\cl Y_a)_{\infty} = \cl C$. We investigate the connection between the two problems in Theorems~\ref{thm_sol_CVOP1} and~\ref{thm_sol_CVOP2}. Recall from Definition~\ref{def_CVOP_sol} that for a convex vector optimization an (exact or approximate) solution is an (exact or approximate) infimizer consisting of minimizers. 
Recall also that the direction $c \in \interior C$ appearing in Definition~\ref{def_CVOP_sol} is assumed to be normalized, i.e.~$\Vert c \Vert = 1$.

\begin{theorem}
\begin{enumerate}
\label{thm_sol_CVOP1}
\item If $\bar{X}\subseteq \mathcal{X}$ is an infimizer of~\eqref{CVOP}, then $\bar{S} := \{ (x,y) \; \vert \; x \in \bar{X}, y \in \{\Gamma(x)\} + C \}$ is a solution of the associated projection~\eqref{Pv}.

\item If the problem~\eqref{CVOP} is self-bounded, then the associated projection~\eqref{Pv} is also self-bounded. If, additionally, the problem~\eqref{CVOP} is bounded, then $(\cl Y_a)_{\infty} = \cl C$.

\item  Let~\eqref{CVOP} be bounded or self-bounded and let $\bar{\bar{X}} \subseteq \mathcal{X}$ be a finite $\epsilon$-infimizer of~\eqref{CVOP}. Then, $\bar{\bar{S}} := \{ (x, \Gamma (x)) \; \vert \; x \in \bar{\bar{X}} \}$ is a finite $\epsilon$-solution of the associated projection~\eqref{Pv}.
\end{enumerate}
\end{theorem}
\begin{proof}
\begin{enumerate}
\item The way the set $\bar{S}$ is constructed implies feasibility as well as $\proj_y [\bar{S}] = \Gamma[\bar{X}] + C$.
Since $\bar{X}$ is a infimizer of~\eqref{CVOP}, it follows that $Y_a \subseteq  \mathcal{G} \subseteq  \cl \left( \conv \Gamma[\bar{X}] + C \right)  = \cl \conv \proj_y [\bar{S}]$.

\item Self-boundedness of~\eqref{Pv} follows directly from $\mathcal{G} = \cl Y_a$.  Boundedness of~\eqref{CVOP} implies $(\cl Y_a)_{\infty} = \mathcal{G}_{\infty} = \cl C$, see Lemma~\ref{lemma_self-bd}.

\item For a finite $\epsilon$-infimizer $\bar{\bar{X}}$ of~\eqref{CVOP} it holds $\mathcal{G} \subseteq \conv \Gamma [\bar{\bar{X}}] + \mathcal{G}_{\infty} -\epsilon \{c\}$ (both in the bounded and in the self-bounded case as $C \subseteq \mathcal{G}_{\infty}$). 
Since $\epsilon c \in B_{\epsilon}$ and $\Gamma [\bar{\bar{X}}] = \proj_y [\bar{\bar{S}}]$ it follows
\begin{align*}
Y_a \subseteq \mathcal{G} \subseteq \conv \Gamma [\bar{\bar{X}}] + \mathcal{G}_{\infty} -\epsilon \{c\} \subseteq \conv \proj_y [\bar{\bar{S}}] + (\cl Y_a)_{\infty} + B_{\epsilon}.
\end{align*}  
\end{enumerate}
\end{proof}

Before moving on, we illustrate that an infimizer (or even a solution) $\bar{X}$ of~\eqref{CVOP} can lead towards a solution $\bar{S}$ that solves the associated projection~\eqref{Pv} only 'up to the closure'. Compare this also to Example~\ref{ex_solution}.
\begin{example}
\label{ex_solution2}
Consider a set $\Gamma [\mathcal{X}] = \mathcal{X} = \{ (x_1, x_2) \; : \; x_1^2 + x_2^2 \leq 1 \}$ with the ordering cone $C = \mathbb{R}^2_+$. The set $\bar{X} = \{ (x_1, x_2) \; : \; x_1^2 + x_2^2 = 1, x_1, x_2 < 0 \}$ is a solution of~\eqref{CVOP}. The corresponding set $\bar{S}$ projects onto $\proj_y [\bar{S}] = \bar{X} + \mathbb{R}^2_+$. It is a solution of~\eqref{Pv}, however, the closure in~\eqref{eq4} is essential as $Y_a \neq \conv \proj_y [\bar{S}]$.
\end{example}

Theorem~\ref{thm_sol_CVOP1} shows how to construct (exact or approximate) solutions of the associated projection from the (exact or approximate) infimizers of the vector optimization problem. Now we look at the other direction. We will show that, given an (exact or approximate) solution $\bar{S}$ of~\eqref{Pv}, the set $\bar{X} = \proj_x [\bar{S}]$ is an (exact or approximate) infimizer of~\eqref{CVOP}. This set $\bar{X}$, however, does not need to consist of minimizers, see Example~\ref{ex_CVOP_sol} below.

In the case of approximate solutions, the tolerance depends on the direction $c \in \interior C$. For this purpose denote
\begin{align}
\label{eq19}
\delta_c := \sup \; \{ \delta > 0 \; \vert \; \{c\} + B_{\delta} \subseteq C \}.
\end{align}
The quantity $\delta_c$ is strictly positive ($c$ is an interior element), finite (the cone $C$ is pointed) and depends on the underlying norm. For example, for the standard ordering cone $C = \mathbb{R}^m_+$ and the direction $\Vert \mathbb{1} \Vert^{-1} \mathbb{1}$ we have $\delta_{ \Vert \mathbb{1} \Vert^{-1} \mathbb{1} } =  \Vert \mathbb{1} \Vert^{-1}$.

\begin{theorem}
\label{thm_sol_CVOP2}
\begin{enumerate}
\item If $\bar{S} \subseteq S_a$ is a solution of~\eqref{Pv}, then $\bar{X} := \proj_x [\bar{S}]$ is an infimizer of~\eqref{CVOP}. 

\item If the problem~\eqref{Pv} is self-bounded, then also the problem~\eqref{CVOP} is self-bounded. If, additionally, $(\cl Y_a)_{\infty} = \cl C$, then the problem~\eqref{CVOP} is bounded.

\item Let the problem~\eqref{Pv} be self-bounded and let $\bar{S} \subseteq S_a$ be a finite $\epsilon$-solution of~\eqref{Pv}. Then, $\bar{X} := \proj_x [\bar{S}]$ is a finite $\tilde{\epsilon}$-infimizer of~\eqref{CVOP} for any tolerance $\tilde{\epsilon} > \dfrac{\epsilon}{\delta_c}$. 
\end{enumerate}
\end{theorem}
The following lemma will be used in the proof of Theorem~\ref{thm_sol_CVOP2}.
\begin{lemma}
\label{lemma_A1}
Let $C \subseteq \mathbb{R}^n$ be a solid cone. For any finite collection of points $\{q^{(1)}, \dots, q^{(k)}\} \subseteq \mathbb{R}^n$ there exists a point $q \in \mathbb{R}^n$ such that $q \leq_C q^{(i)}$, i.e.~$q^{(i)} \in q + C$, for all $i = 1, \dots, k$.
\end{lemma}

\begin{proof}
For a solid cone $C$ there exists $c \in \interior C$. Without loss of generality assume that $c$ is scaled in such a way that $\{c\} + B_1 \subseteq C$. The set $\{c\} + B_1$ generates a convex solid cone contained within $C$.

Consider two points $q^{(1)}, q^{(2)} \in \mathbb{R}^n$. Define $q:= q^{(1)} - \Vert q^{(1)} - q^{(2)} \Vert \left( c + \frac{q^{(1)} - q^{(2)}}{\Vert q^{(1)} - q^{(2)} \Vert} \right) = q^{(2)} - \Vert q^{(1)} - q^{(2)} \Vert c$. Then $q^{(1)}, q^{(2)} \in q + \cone (\{c\} + B_1) \subseteq \{q\} + C$. 
For more than two points use an induction argument.
  \end{proof}

\begin{proof}[of Theorem~\ref{thm_sol_CVOP2}.]
\begin{enumerate}
\item For the solution $\bar{S}$ and the set $\bar{X}$ it holds
$
Y_a \subseteq \cl \left( \conv \proj_y [\bar{S}] \right) \subseteq \cl \left( \conv \Gamma [\bar{X}] + C \right).
$
Since the right-hand side is closed and $\mathcal{G} = \cl Y_a$, the set $\bar{X}$ is an infimizer of~\eqref{CVOP}.

\item A self-bounded problem~\eqref{Pv} satisfies $Y_a \neq \mathbb{R}^m$ and there exist $y^{(1)}, \dots, y^{(k)} \in \mathbb{R}^m$ such that
\begin{align*}
\Gamma [\mathcal{X}] + C = Y_a \subseteq \conv \{y^{(1)}, \dots, y^{(k)} \} + (\cl Y_a)_{\infty}.
\end{align*}
Since the set $Y_a$ is convex it follows that also $\mathcal{G} = \cl Y_a  \neq \mathbb{R}^m$.
Since the recession cone $(\cl Y_a)_{\infty}$ contains the solid cone $C$, there exists a point $q \in \mathbb{R}^m$ such that $\conv \{y^{(1)}, \dots, y^{(k)} \} + (\cl Y_a)_{\infty} \subseteq \{q\} + (\cl Y_a)_{\infty}$, see Lemma~\ref{lemma_A1}. As $(\cl Y_a)_{\infty} = \mathcal{G}_{\infty}$ and the shifted cone is closed, we obtain $\mathcal{G} \subseteq \{q\} + \mathcal{G}_{\infty}.$ The second claim follows from Lemma~\ref{lemma_self-bd}.

\item 
The finite $\epsilon$-solution $\bar{S}$ of~\eqref{Pv} satisfies
\begin{align*}
\Gamma [\mathcal{X}] + C = Y_a \subseteq \conv \proj_{y} [\bar{S}] + (\cl Y_a)_{\infty} + B_{\epsilon}.
\end{align*}
By~\eqref{eq19} for all $0 < \delta < \delta_c$ it holds $B_{\epsilon} \subseteq -\frac{\epsilon}{\delta} \{c\} + C$. 
For the set $\conv \proj_y [\bar{S}] $ it holds $\conv \proj_y [\bar{S}] \subseteq \conv \Gamma [\bar{X}] + C$. Since $(\cl Y_a)_{\infty} = \mathcal{G}_{\infty}$ is a convex cone containing $C$, we obtain
\begin{align*}
\Gamma [\mathcal{X}] + C \subseteq \conv \Gamma [\bar{X}] + \mathcal{G}_{\infty} -\frac{\epsilon}{\delta} \{c\}.
\end{align*}
Since $\bar{X}$ is finite, the set $\conv \Gamma [\bar{X}]$ is closed convex with a trivial recession cone. The sets $\conv \Gamma [\bar{X}]$ and $\mathcal{G}_{\infty}$ then fulfill all assumptions of the Corollary 9.1.2 of~\cite{R70} and, therefore, the set $\conv \Gamma [\bar{X}] + \mathcal{G}_{\infty}$ is closed. This gives 
\begin{align*}
\mathcal{G} \subseteq \conv \Gamma [\bar{X}] + \mathcal{G}_{\infty} -\frac{\epsilon}{\delta} \{c\}
\end{align*}
for all $0 < \delta < \delta_c$.
In the self-bounded case this proves the claim. In the bounded case with $\mathcal{G}_{\infty} = \cl C$ the relation $\cl C \subseteq -\tilde{\delta} \{c\} + C$ with an arbitrarily small $\tilde{\delta} > 0$ gives the desired result.
  
\end{enumerate}
\end{proof}

Next we provide an example of (both exact and approximate) solutions of~\eqref{Pv}, which yield infimizers, but not solutions of~\eqref{CVOP}.
\begin{example}
\label{ex_CVOP_sol}
Consider the trivial example of $\min x \text{ s.t. } x \geq 0$ with $\mathcal{G} = \Gamma [\mathcal{X}] = [0, \infty )$ and the associated feasible set $S_a = \{ (x, y) \; \vert \; y \geq x \geq 0 \}$. The set $\bar{S}_1 =  \{ (x, y) \; \vert \; y \geq x > 0 \}$ is a solution of~\eqref{Pv}. The corresponding $\bar{X}_1 = \proj_x \bar{S}_1 = (0, \infty)$ is an infimizer, but not a solution of~\eqref{CVOP}. 

The situation is similar for approximate solutions: Fix $\epsilon > 0$. The set $\bar{S}_2 = \{ (\epsilon, \epsilon)\}$ is a finite $\epsilon$-solution of the convex projection. The set $\bar{X}_2 = \proj_x \bar{S}_2 = \{\epsilon\}$ is a finite $\epsilon$-infimizer, but it does not consist of minimizers.
\end{example}
In the polyhedral case~\cite{LW16} it is possible (under an assumption on the lineality space of $\mathcal{G}$) to obtain a solution of the vector optimization problem by removing non-minimal points from the solution of the associated projection. This is not possible for either the exact solution $\bar{S}_1$ or the approximate solution $\bar{S}_2$ in the above example, as both consist of non-minimal points only. However, for exact solutions, we can formulate conditions under which an (exact) solution of~\eqref{CVOP}  can be constructed from a solution of~\eqref{Pv}, see Lemma~\ref{lemma_construct_sol} below.
It is, however, of theoretical interest rather than of practical use.
\begin{lemma}
\label{lemma_construct_sol}
Assume that a solution of~\eqref{CVOP} exists and let the solution $\bar{S}$ of~\eqref{Pv} satisfy $Y_a = \conv \proj_y [\bar{S}]$. Then, $\bar{X} := \{ x \; \vert \; (x,y) \in \bar{S}, y \not\in \Gamma [\mathcal{X}] + C \backslash \{0\} \}$ is a solution of~\eqref{CVOP}.
\end{lemma}
\begin{proof}
Denote $\bar{S}_0 := \{ (x,y) \in \bar{S} \; \vert \; y \not\in \Gamma [\mathcal{X}] + C \backslash \{0\} \}$.
Let $\bar{X}_1 \subseteq \mathcal{X}$ be some solution of~\eqref{CVOP}. 
By Lemma~\ref{lemmaYa} and by assumption it holds
\begin{align}
\label{eq18}
\Gamma [\bar{X}_1] \subseteq \Gamma [\mathcal{X}] \subseteq \Gamma [\mathcal{X}] + C = Y_a = \conv \proj_y [\bar{S}].
\end{align}
We prove that $\Gamma [\bar{X}_1] \subseteq \conv \proj_y [\bar{S}_0]$: Take arbitrary $\bar{x}_1 \in \bar{X}_1$, according to~\eqref{eq18} we have $\Gamma [\bar{x}_1] = \sum\limits_{y_i \in \proj_y [\bar{S}]} \lambda_i y_i$, where $\sum \lambda_i = 1$ and all $\lambda_i \geq 0$. Assuming that for some $\lambda_i > 0$ we have  $y_i \in \Gamma [\mathcal{X}] + C \backslash \{0\}$ contradicts $\bar{x}_1 \in \bar{X}_1$ being a minimizer. The relation $\Gamma [\bar{X}_1] \subseteq \conv \proj_y [\bar{S}_0]$ proves the claim as $\bar{X} = \proj_x [\bar{S}_0]$.
  \end{proof}

Both assumption in Lemma~\ref{lemma_construct_sol} are essential. The condition $Y_a = \conv \proj_y [\bar{S}]$ in effect guarantees that the (exact) solution $\bar{S}$ of~\eqref{Pv} can be reduced to only minimal points. Compare this to Example~\ref{ex_CVOP_sol}. Unfortunately, one cannot not expect a practical version of Lemma~\ref{lemma_construct_sol} for approximate solutions. Since approximate solutions contain an approximation error, the condition $Y_a = \conv \proj_y [\bar{S}]$ is not feasible there. We would need to assume that each element of the approximate solution is either minimal or redundant. Such requirement is, however, almost tautological.

In the polyhedral case, the equivalence between polyhedral projection, multi-objective linear programming and vector linear programming makes it possible to construct a multi-objective problem corresponding to an initial vector linear program. One could do this also in the convex case: combine the associated projection~\eqref{Pv} and the ideas of the previous section to construct the multi-objective problem 
\begin{align}
\label{MOCP2}
\min 
\begin{pmatrix}
y \\
- \trans{\mathbf{1}} y
\end{pmatrix} 
\text{ with respect to } \leq_{\mathbb{R}^{m+1}_+} \text{ subject to } (x, y) \in S_a
\end{align}
associated to~\eqref{CVOP}. As in the linear case, the dimension of the objective space of the multi-objective problem~\eqref{MOCP2} is only one higher than the original problem~\eqref{CVOP}. The results of the last two sections can be combined to establish a connection between~\eqref{CVOP} and~\eqref{MOCP2}. Unfortunately, doing so combines all drawbacks of these results. Most importantly, it involves only (exact or approximate) infimizers of~\eqref{CVOP}, but not solutions. For completeness we list these combined results in the following corollary, the upper image of~\eqref{MOCP2} is denoted $\mathcal{P}_a = \cl (P[S_a] + \mathbb{R}^{m+1}_+)$.
\begin{corollary} $\,$
\begin{enumerate}
\item The upper images of~\eqref{CVOP} and~\eqref{MOCP2} are connected via $\mathcal{P}_a = Q [\mathcal{G}] + \mathbb{R}^{m+1}_+$.

\item Problem~\eqref{CVOP} is self-bounded if and only if \eqref{MOCP2} is self-bounded. Additionally, \eqref{CVOP} is bounded if and only if \eqref{MOCP2} is self-bounded with $(\mathcal{P}_a)_{\infty} = Q[\cl C] + \mathbb{R}^{m+1}_+$.

\item If $\bar{S} \subseteq S_a$ is a solution of~\eqref{MOCP2}, then $\bar{X} := \proj_x [\bar{S}]$ is an infimizer of~\eqref{CVOP}.

\item If $\bar{X} \subseteq \mathcal{X}$ is an infimizer of~\eqref{CVOP}, then $\bar{S} := \{ (x,y) \; \vert \; x \in \bar{X}, y \in \{\Gamma(x)\} + C \}$ is a solution of~\eqref{MOCP2}.

\item If the problem~\eqref{CVOP} is bounded or self-bounded and $\bar{X} \subseteq \mathcal{X}$ is a finite $\epsilon$-infimizer of~\eqref{CVOP}, then $\bar{S} := \{ (x, \Gamma (x)) \; \vert \; x \in \bar{X} \}$ is a finite $ ( \underline{\kappa} \cdot  \epsilon)$-solution of~\eqref{MOCP2}.

\item If~\eqref{MOCP2} is self-bounded and $\bar{S} \subseteq S_a$ is a finite $\epsilon$-solution of~\eqref{MOCP2}, then $\bar{X} := \proj_x [\bar{S}]$ is a finite $\xi$-infimizer of~\eqref{CVOP} for any tolerance $\xi > \dfrac{ \overline{\kappa} }{\delta_c}\cdot \epsilon$.
\end{enumerate}
\end{corollary}

\section{Solution concept according to~\cite{DLSW21}}
\label{appendix_DLSW21}

In the previous sections of this paper we work with the definition of a finite $\epsilon$-solution of a bounded (or self-bounded) convex vector optimization problem from~\cite{LRU14, U18}. The idea there is to shift the (inner) approximation in a fixed direction to cover the full upper images (i.e.~obtain an outer approximation), see~\eqref{eq16}. In \cite{DLSW21}, a slightly different definition of a finite $\epsilon$-solution of a bounded~\eqref{CVOP} is proposed.  Here the idea is to bound the Hausdorff distance between the inner approximation and the upper image by the given tolerance. To distinguish between the two definitions we will speak about \textit{\cite{LRU14}-finite $\epsilon$-solutions} and \textit{\cite{DLSW21}-finite $\epsilon$-solutions}. In this section we revisit our main results, Theorems~\ref{thm_CPsolMOCP} and~\ref{thm_MOCPsolCP} from Section~\ref{sec_solutions} and Theorems~\ref{thm_sol_CVOP1} and~\ref{thm_sol_CVOP2} from Section~\ref{sec_CVOP_CP}, under this alternative definition.

Given a norm $\Vert \cdot \Vert$, the Hausdorff distance between two sets $A_1 \subseteq \mathbb{R}^n$ and $A_2 \subseteq \mathbb{R}^n$ is defined as
\begin{align*}
d_H (A_1, A_2) = \max \left\lbrace \sup_{a_1 \in A_1} \inf_{a_2 \in A_2} \Vert a_1 - a_2 \Vert,  \sup_{a_2 \in A_2} \inf_{a_1 \in A_1} \Vert a_1 - a_2 \Vert\right\rbrace.
\end{align*}
One easily verifies that for sets $A_1 \subseteq A_2$ it holds $d_H  (A_1, A_2) = \sup_{a_2 \in A_2} \inf_{a_1 \in A_1} \Vert a_1 - a_2 \Vert$ and the condition $d_H  (A_1, A_2) \leq \epsilon$ is equivalent to the condition $A_2 \subseteq A_1 + B_{\epsilon}$. The following definition is proposed in~\cite{DLSW21} for a bounded~\eqref{CVOP}: A nonempty finite set $\bar{X}\subseteq \mathcal{X}$ is a \cite{DLSW21}-finite $\epsilon$-infimizer of~\eqref{CVOP} if
\begin{align}
\label{eqA01}
d_H \left( \mathcal{G}, \conv \Gamma [\bar{X}] + C \right) \leq \epsilon.
\end{align}
We now address the connection between these two definitions of finite $\epsilon$-infimizers. This was done in~\cite[Proposition 3.5]{DLSW21} under the Euclidean norm. Here, we adapt that proof to the $p$-norm.
\begin{lemma}
\label{lemma_two_infiizers}
\begin{enumerate}
\item If $\bar{X}$ is a \cite{LRU14}-finite $\epsilon$-infimizer (see Definition~\ref{def_CVOP_sol}) of a bounded problem~\eqref{CVOP}, then it is also a \cite{DLSW21}-finite $\epsilon$-infimizer of~\eqref{CVOP}.

\item If $\bar{X}$ is a \cite{DLSW21}-finite $\epsilon$-infimizer of a bounded problem~\eqref{CVOP}, then it is also a \cite{LRU14}-finite $(k \cdot \epsilon)$-infimizer of~\eqref{CVOP}, where
\begin{align} 
\label{eqA04}
k = \frac{ 1 }{  \min \left\lbrace \trans{w} c \; \vert \;  w \in C^+, \| w \|_q = 1 \right\rbrace }.
\end{align}
Here $q$ is given by $\frac{1}{p} + \frac{1}{q} = 1$ for $p \in (1, \infty)$, respectively $q = \infty$ for $p = 1$ and $q = 1$ for $p = \infty$.
\end{enumerate}
\end{lemma}
\begin{proof}
The first claim trivially follows from the fact that the element $c \in \interior C$ is assumed to be normalized, so~\eqref{eq16} implies $\mathcal{G} \subseteq \conv \Gamma [\bar{X}] + C + B_{\epsilon}$.

We adapt the proof of \cite[Proposition 3.5]{DLSW21} to the case of a $p$-norm, we mainly highlight the changes needed because of the different norm, we refer the reader to~\cite{DLSW21} for details. First, note that H\"older's inequality implies $\vert \trans{w} c \vert \leq  \| w\|_q \| c \|_p$, so $k \geq 1$. The closed convex set $\conv \Gamma [\bar{X}] + C$ admits a representation
\begin{align*}
\conv \Gamma [\bar{X}] + C = \bigcap\limits_{i \in I} \left\lbrace y \in \mathbb{R}^m \vert \trans{w}_i y \geq \gamma_i \right\rbrace
\end{align*}
for some index set $I$, scalars $\gamma_i \in \mathbb{R}$ and vectors $w_i \in C^+ \setminus \{0\}$ that without loss of generality satisfy $\|w_i\|_q = 1$. Take arbitrary $g \in \mathcal{G}$ and define $k_g := \inf \left\lbrace t \geq 1 \right. \vert g + \epsilon t c \in \left. \conv \Gamma [\bar{X}] + C \right\rbrace$. If $k_g = 1$, then $k_g \leq k$ holds trivially. For $k_g > 1$ there exists an index $j \in I$ such that $\trans{w}_j \left( g + \epsilon k_g c \right) = \gamma_j$, which allows us to express $k_g$ as $k_g = \frac{\gamma_j - \trans{w}_j g}{\epsilon \trans{w}_j c}$. Now we show that $\gamma_j - \trans{w}_j g \leq \epsilon$: Since $d_H \left( \mathcal{G}, \conv \Gamma [\bar{X}] + C \right) \leq \epsilon$ holds, there exists $u \in \mathbb{R}^m$ with $\|u \|_p \leq \epsilon$ such that $g + u \in \conv \Gamma [\bar{X}] + C$. Assuming $\gamma_j - \trans{w}_j g >  \epsilon$ would yield $\trans{w}_j \left( g+u \right) \geq \gamma_j > \trans{w}_j g + \epsilon \geq \trans{w}_j g + \|u\|_p \| w_j \|_q$, which contradicts the H\"older's inequality. This yields
\begin{align*}
k_g  \leq \frac{1}{\trans{w}_j c} \leq k,
\end{align*}
which proves the claim.
  \end{proof}

Let us consider the convex projection~\eqref{CP} and the associated multi-objective problem~\eqref{MOCP} in the bounded case. Note that for the multi-objective problem~\eqref{MOCP} 
with $C = \mathbb{R}^{m+1}_+$ and $c = \|\mathbb{1}\|^{-1} \mathbb{1}$ (recall that $\| \cdot \|$ denotes the $p$-norm) we have $ \min \{ \|\mathbb{1}\|^{-1} \trans{w} \mathbb{1} \; \vert \;  w \in \mathbb{R}^{m+1}_+, \| w \|_q = 1 \}  = \|\mathbb{1}\|^{-1}$, which follows by considering the problem $\min \trans{\mathbb{1}} w \text{ s.t. }  \| w \|_q = 1, w \geq 0$. 
Note further, that condition~\eqref{eq5} that defines a finite $\epsilon$-solution of~\eqref{CP} can be equivalently stated as $d_H (Y, \conv \proj_y [\bar{S}]) \leq \epsilon.$ 
Recall that a nonempty finite set $\bar{S} \subseteq S$ is a \cite{LRU14}-finite $\epsilon$-solution of~\eqref{MOCP} if it holds
\begin{align*}
\mathcal{P} \subseteq \conv P[\bar{S}] + \mathbb{R}^{m+1}_+ - \epsilon\{ \Vert \mathbb{1} \Vert^{-1} \mathbb{1}\};
\end{align*}
and a nonempty finite set $\bar{S} \subseteq S$ is a \cite{DLSW21}-finite $\epsilon$-solution of~\eqref{MOCP} if
\begin{align*}
d_H \left( \mathcal{P}, \conv P[\bar{S}] + \mathbb{R}^{m+1}_+ \right) \leq \epsilon.
\end{align*}

First, we revisit the question studied in Theorem~\ref{thm_CPsolMOCP} under the \cite{DLSW21}-solution concept.
We will show in Proposition~\ref{lemma_E1} below that an approximate solution of the projection~\eqref{CP} is also an approximate solution of~\eqref{MOCP} in the \cite{DLSW21} sense, where the multiplier is given by the norm $\Vert Q \Vert$. The value of $\Vert Q \Vert$ is deduced in the following lemma.

\begin{lemma}
\label{lemma_normQ}
The operator norm $\Vert Q \Vert$ of the matrix $Q$ induced by the vector-$p$-norm is $\Vert Q \Vert = \left( m^{p-1} + 1 \right)^{\frac{1}{p}}$ for $p \in [1, \infty)$ and $\Vert Q \Vert = m$ for $p=\infty$.
\end{lemma}
\begin{proof}
The operator norm is defined as $\Vert Q \Vert = \max\limits_{\|x\| = 1} \| Q x \|$, which can be computed by considering~the problem $\max \trans{\mathbf{1}} x \text{ s.t. } \Vert x \Vert = 1$.
  \end{proof}

\begin{proposition}
\label{lemma_E1}
Let the problem~\eqref{CP} be bounded and let $p \in [1, \infty]$ be fixed. If the set $\bar{S} \subseteq S$ is a finite $\epsilon$-solution of~\eqref{CP}, then it is a \cite{DLSW21}-finite $\left( \Vert Q \Vert \cdot \epsilon \right)$-solution of~\eqref{MOCP}. 
\end{proposition}
\begin{proof}
First we show that $d_H \left( \mathcal{P}, \conv P[\bar{S}] + \mathbb{R}^{m+1}_+ \right) \leq d_H \left( P[S], \conv P[\bar{S}] \right)$. Since a closure does not influence the Hausdorff distance, the left-hand side is $d_H \left( P[S] + \mathbb{R}^{m+1}_+, \conv P[\bar{S}] + \mathbb{R}^{m+1}_+ \right)$. Take arbitrary $q \in P[S]$ and $r \in \mathbb{R}^{m+1}_+$. Since the norm satisfies the triangle inequality we have 
\begin{align*}
\inf_{\substack{\bar{q} \in \conv P[\bar{S}],\\ \bar{r} \in \mathbb{R}^{m+1}_+}} \| (q + r) - (\bar{q} + \bar{r}) \| \leq \inf_{\substack{\bar{q} \in \conv P[\bar{S}],\\ \bar{r} \in \mathbb{R}^{m+1}_+}} \| q - \bar{q} \| + \| r - \bar{r} \| \leq \sup_{q \in P[S]} \inf_{\bar{q} \in \conv P[\bar{S}]} \| q - \bar{q} \|.
\end{align*}
As the points $q \in P[S]$ and $r \in \mathbb{R}^{m+1}_+$ are arbitrary, this gives the desired inequality.
Now let $\bar{S}$ be a finite $\epsilon$-solution of~\eqref{CP}. Hence, it satisfies
\begin{align}
\label{eqE03}
\sup_{y \in Y} \inf_{\bar{y} \in \conv \proj_y [\bar{S}]} \Vert y - \bar{y} \Vert \leq \epsilon.
\end{align}
Keep in mind that the sets of interest satisfy $P[S] = Q[Y]$ and $P[\bar{S}] = Q  [\proj_y [\bar{S}]]$ , see Lemma~\ref{lemma_minimizer}~(2) and its proof. The induced matrix norm $\Vert \cdot \Vert : \mathbb{R}^{(m+1) \times m} \rightarrow \mathbb{R}$ is consistent, so this gives us the desired result,
\begin{align*}
d_H \left( \mathcal{P}, \conv P[\bar{S}] + \mathbb{R}^{m+1}_+ \right) &\leq  \sup_{q \in P[S]} \inf_{\bar{q} \in \conv P[\bar{S}]} \| q - \bar{q} \| = \sup_{y \in Y} \inf_{\bar{y} \in \conv \proj_y [\bar{S}]} \| Qy - Q\bar{y} \| \\ 
&\leq \sup_{y \in Y} \inf_{\bar{y} \in \conv \proj_y [\bar{S}]} \| Q\| \cdot \| y -\bar{y} \| = \|Q\| \cdot d_H \left( Y, \conv \proj_y [\bar{S}] \right) \leq \|Q\| \cdot \epsilon.
\end{align*}
  \end{proof}

Theorem~\ref{thm_CPsolMOCP} and Proposition~\ref{lemma_E1} provide us with two different multipliers, one for each solution concept.  The two solution concepts are, however, not unrelated, as was shown in Lemma~\ref{lemma_two_infiizers}. 
With this observation, one could now also apply Theorem~\ref{thm_CPsolMOCP} to a \cite{DLSW21}-type solution and Proposition~\ref{lemma_E1} to a \cite{LRU14}-type solution. The resulting multipliers are listed in Table~\ref{tab3}. One can observe that the multipliers applied directly to the solution concepts used in Theorem~\ref{thm_CPsolMOCP}, respectively in Proposition~\ref{lemma_E1}, are better than those obtained when going through the other solution concept first and applying Lemma~\ref{lemma_two_infiizers}.

\begin{table}
\begin{center}
\caption{\label{tab3}  Multipliers for an approximate solution of~\eqref{CP} to solve~\eqref{MOCP}: 
Theorem~\ref{thm_CPsolMOCP} provides multiplier $\underline{\kappa}$ for the \cite{LRU14}-type solution and Proposition~\ref{lemma_E1} provides multiplier $\| Q \|$ for the \cite{DLSW21}-type solution. Lemma~\ref{lemma_two_infiizers} implies that $\underline{\kappa}$ is a feasible multiplier for the \cite{DLSW21}-type solution and $\| Q \| \cdot k$ is a feasible multiplier for the \cite{LRU14}-type solution, where $k$ is given in \eqref{eqA04}. 
} 
\begin{tabular}{c c c c}
\hline \hline
									& Theorem~\ref{thm_CPsolMOCP}						&		& Proposition~\ref{lemma_E1} \\ \hline
\cite{LRU14}-type solution	& $m^{\frac{p-1}{p}} \cdot (m+1)^{\frac{1}{p}}$		& $<$	& $\left( m^{p-1} + 1 \right)^{\frac{1}{p}} \left( m+1 \right)^{\frac{1}{p}}$ \\
\cite{DLSW21}-type solution		& $m^{\frac{p-1}{p}} \cdot (m+1)^{\frac{1}{p}}$		& $>$	& $\left( m^{p-1} + 1 \right)^{\frac{1}{p}}$ \\ \hline \hline
\end{tabular}
\end{center}
\end{table}

Second, we revisit the question studied in Theorem~\ref{thm_MOCPsolCP}, but now under the \cite{DLSW21}-solution concept. We show in Proposition~\ref{lemma_A2AL} below, that an approximate solution of~\eqref{MOCP} in the \cite{DLSW21} sense is also an approximate solution of the projection~\eqref{CP} and deduce the multiplier. This multiplier was independently obtained in~\cite{LZS21} under the Euclidean norm in a closely related setting.

\begin{proposition}
\label{lemma_A2AL}
Let the problem~\eqref{MOCP} be bounded and let $p \in [1, \infty]$ be fixed. If the set $\bar{S} \subseteq S$ is a \cite{DLSW21}-finite $\epsilon$-solution of~\eqref{MOCP}, then it is a finite $\left( \| Q \| \cdot \epsilon \right)$-solution of~\eqref{CP}. 
\end{proposition}
To prove the proposition we will use the following result.
\begin{lemma}
\label{lemma_A5}
The optimization problem
\begin{align}
\label{prob_A01}
\begin{split}
\max\limits_{r, b \in \mathbb{R}^{m+1} } \; & \| r - b\| \\
\text{s.t. } &\|b \| \leq \epsilon, \\
& r \geq 0, \\
& \trans{\mathbb{1}} (r-b) = 0
\end{split}
\end{align}
for $\epsilon > 0$, 
has an optimal objective value $\|r^* - b^* \| =  \epsilon \cdot \Vert Q \Vert$.
\end{lemma}
\begin{proof}
First, note that we can restrict ourselves to vectors $b$ satisfying $b \geq 0$. Consider any feasible pair of vectors $(r, b)$. Construct a new pair of vectors $(\bar{r}, \bar{b})$ by $\bar{b}_i = (b_i)^+$ and $\bar{r}_i = r_i + (b_i)^-$ for all $i = 1, \dots, m+1$. The new pair $(\bar{r}, \bar{b})$ is also feasible for the problem and gives the same objective value as the original pair $(r, b)$.  

Second, fix some vector $b \geq 0$. We now solve the problem
\begin{align}
\label{prob_A02}
\begin{split}
\max\limits_{r \in \mathbb{R}^{m+1}} \; & \| r - b \| \\
\text{s.t. } &\trans{\mathbb{1}} r  \leq \trans{\mathbb{1}} b, \\
& r  \geq 0. 
\end{split}
\end{align}
Note that we relaxed the equality constraint to an inequality. However, for the optimal solution the constraint will be satisfied as an equality. Just as problem~\eqref{op} considered within the proof of Theorem~\ref{thm_MOCPsolCP}, problem~\eqref{prob_A02} maximizes a convex objective over a compact polyhedron. Thus, the maximum is attained in a vertice and the optimal solution of problem~\eqref{prob_A02} is $r^* (b) = (\trans{\mathbb{1}} b) \cdot e^{(i^*)}$, where $i^* \in \argmin_{i=1, \dots, m+1} b_i $.

Finally, we return to problem~\eqref{prob_A01}. By our first argument, it suffices to look at nonnegative vectors $b$. According to the above results for problem~\eqref{prob_A02}, problem~\eqref{prob_A01} simplifies to
\begin{align*}
\max \; & \| (\trans{\mathbb{1}} b) \cdot e^{(1)}  - b \| \\
\text{s.t. } &\|b\| \leq \epsilon, \\
& b \geq 0,
\end{align*}
where, without loss of generality, one can assume that the first coordinate of the vector $b$ is the smallest. Consider $p \in [1, \infty)$. 
The objective function $\| (\trans{\mathbb{1}} b) \cdot e^{(1)}  - b \| = \left( \left(\sum\limits_{i=2}^{m+1} b_i \right)^p + \sum\limits_{i=2}^{m+1} b_i^p \right)^{\frac{1}{p}}$ and the constraint $\sum\limits_{i=2}^{m+1} b_i^p \leq \epsilon - b_1^p$ yield an optimal solution $b^* = \epsilon \cdot (0, m^{-\frac{1}{p}}, \dots, m^{-\frac{1}{p}})$. Hence, $r^* = (\trans{\mathbb{1}} b^*) \cdot e^{(1)} = \epsilon \cdot (m^{\frac{p-1}{p}}, 0, \dots, 0)$ and $\|r^* - b^* \| = \epsilon \cdot \left( m^{p-1} + 1 \right)^{\frac{1}{p}} = \epsilon \cdot \Vert Q \Vert$. Finally, consider $p = \infty$. Thanks to non-negativity of $b$ the objective equals $\| (\trans{\mathbb{1}} b) \cdot e^{(1)}  - b \| = \sum\limits_{i=2}^{m+1} b_i$, so an optimal solution is e.g.~$b^* = \epsilon \cdot (0, 1, \dots, 1 )$. Hence, $r^* = \epsilon \cdot m \cdot e^{(1)}$ and $\|r^* - b^* \| = \epsilon \cdot m$.
  \end{proof}

\begin{proof}[Proof of Proposition~\ref{lemma_A2AL}.]
The definition of a \cite{DLSW21}-finite $\epsilon$-solution and $P[S] \subseteq \mathcal{P} \cap Q [\mathbb{R}^m]$ imply
\begin{align*}
P[S] \subseteq \conv P[\bar{S}] + \left( \mathbb{R}^{m+1}_+ + B_{\epsilon} \right) \cap Q [\mathbb{R}^m].
\end{align*}
In order to bound the distance $d_H \left( P[S], \conv P[\bar{S}] \right)$ we need to contain the set $\left( \mathbb{R}^{m+1}_+ + B_{\epsilon} \right) \cap Q [\mathbb{R}^m]$ within a ball. Finding the appropriate radius corresponds to solving problem~\eqref{prob_A01} (use symmetry of the ball $B_{\epsilon}$ to fit the sign convention).
According to Lemma~\ref{lemma_A5}, the optimal objective value of this problem is $\epsilon \cdot \|Q\|$. Thus, one obtains $P[S] \subseteq \conv P[\bar{S}] + B_{\epsilon \cdot  \|Q\|}$. Applying $\proj_{-1}$ on both sides yields $Y \subseteq \conv \proj_y [\bar{S}] + B_{\epsilon \cdot  \|Q\|}$ and thus the desired multiplier.
  \end{proof}

Again, the connection between a \cite{DLSW21}-type solution and a \cite{LRU14}-type solution makes it possible to use Theorem~\ref{thm_MOCPsolCP} and Proposition~\ref{lemma_A2AL} indirectly also for the other type of solution. The resulting multipliers are listed in Table~\ref{tab4}. Again, the direct results of Theorem~\ref{thm_MOCPsolCP} and Proposition~\ref{lemma_A2AL} are better than the results when going through the other solution concept first. 

\begin{table}
\begin{center} 
\caption{\label{tab4}  Multipliers for an approximate solution of~\eqref{MOCP} to solve~\eqref{CP}: 
Theorem~\ref{thm_MOCPsolCP} provides multiplier $\overline{\kappa}$ for the \cite{LRU14}-type solution and Proposition~\ref{lemma_A2AL} provides multiplier $\| Q \|$ for the \cite{DLSW21}-type solution. Lemma~\ref{lemma_two_infiizers} implies that $\overline{\kappa} \cdot k$ is a feasible multiplier for the \cite{DLSW21}-type solution and $\| Q \|$ is a feasible multiplier for the \cite{LRU14}-type solution, where $k$ is given in \eqref{eqA04}. 
} 
\begin{tabular}{c c c c}
\hline \hline
									& Theorem~\ref{thm_MOCPsolCP}						&		& Proposition~\ref{lemma_A2AL} \\ \hline
\cite{LRU14}-type solution	& $\left( \frac{m^p + m - 1 }{m+1} \right)^{\frac{1}{p}}$		& $<$	& $\left( m^{p-1} + 1 \right)^{\frac{1}{p}}$ \\
\cite{DLSW21}-type solution		& $\left( m^p + m - 1 \right)^{\frac{1}{p}}$		& $>$	& $\left( m^{p-1} + 1 \right)^{\frac{1}{p}}$ \\ \hline \hline
\end{tabular}
\end{center}
\end{table}

Finally, let us consider a bounded problem~\eqref{CVOP} and the associated projection~\eqref{Pv} and deduce the results that correspond to Theorem~\ref{thm_sol_CVOP1}~(3.) and Theorem~\ref{thm_sol_CVOP2}~(3.) but for the solution concept from~\cite{DLSW21}.

\begin{proposition}
\begin{enumerate}

\item Let the problem~\eqref{CVOP} be bounded. If $\bar{X} \subseteq \mathcal{X}$ is a \cite{DLSW21}-finite $\epsilon$-infimizer of~\eqref{CVOP}, then $\bar{S} := \{ (x, \Gamma (x)) \vert x \in \bar{X} \}$ is a finite $\epsilon$-solution of the convex projection~\eqref{Pv}.

\item Let the associated convex projection~\eqref{Pv} be self-bounded with $(\cl Y_a)_{\infty} = \cl C$. If $\bar{S}$ is  a finite $\epsilon$-solution of~\eqref{Pv}, then $\bar{X} := \proj_x [\bar{S}]$ is a \cite{DLSW21}-finite $\xi$-infimizer of~\eqref{CVOP} for all $\xi > \epsilon$.

\end{enumerate}
\end{proposition}

\begin{proof}
\begin{enumerate}

\item A \cite{DLSW21}-finite $\epsilon$-infimizer $\bar{X}$ satisfies
\begin{align*}
\mathcal{G} \subseteq \conv \Gamma [\bar{X}] + C + B_\epsilon.
\end{align*}
Since $Y_a \subseteq \mathcal{G}, C \subseteq (\cl Y_a)_{\infty}$ and $\proj_x [\bar{S}] = \Gamma [\bar{X}]$, it follows that
\begin{align*}
Y_a \subseteq \conv \proj_x [\bar{S}] + (\cl Y_a)_{\infty} + B_\epsilon.
\end{align*}

\item The assumption implies a bounded problem~\eqref{CVOP}, see Theorem~\ref{thm_sol_CVOP2}~(2). A finite $\epsilon$-solution $\bar{S}$ satisfies
\begin{align*}
Y_a \subseteq \conv \proj_y [\bar{S}] + (\cl Y_a)_{\infty} + B_\epsilon.
\end{align*}
Since $Y_a = \Gamma [\mathcal{X}] + C$ and $\proj_y[\bar{S}] \subseteq \Gamma [\bar{X}] + C$ we get 
\begin{align*}
\Gamma [\mathcal{X}] + C \subseteq \conv \Gamma [\bar{X}] + \cl C + B_\epsilon.
\end{align*}
For a solid cone $C$ it holds $\cl C \subseteq -\delta_1 c + C$ for  a (without loss of generality normalized) interior point $c \in \interior C$ and any $\delta_1 > 0$. This gives us $\Gamma [\mathcal{X}] + C \subseteq \conv \Gamma [\bar{X}] +  C - \delta_1 \{c\} + B_\epsilon \subseteq \conv \Gamma [\bar{X}] +  C  + B_{\epsilon+\delta_1}$ for an arbitrarily small $\delta_1 > 0$. Applying the closure, we obtain $\mathcal{G} \subseteq \cl \left( \conv \Gamma [\bar{X}] + C + B_{\epsilon+\delta_1} \right) $ and an additional arbitrarily infinitesimal increase by $\delta_2 > 0$ of the tolerance to $\xi = \epsilon + \delta_1 + \delta_2$ covers the closure. Since both $\delta_1$ and $\delta_2$ are arbitrarily small, any tolerance $\xi > \epsilon$ is achievable.

\end{enumerate}
\end{proof}

\section*{Acknowledgement}
We would like to thank Andreas L\"ohne for bringing our attention to the projection problem and for pointing out the connection between parametrized linear vector optimization problems and polyhedral projections during a research visit at the Institute for Statistics and Mathematics
at Vienna University of Economics and Business in Summer 2016. Furthermore, we would like to thank him for a very helpful hint concerning the proof of Lemma~\ref{lemma_recession}. 
The authors would like to thank two anonymous referees for useful comments and suggestions that helped improving the manuscript.

\bibliographystyle{alpha}
\bibliography{biblioCP}

\end{document}